\documentclass[12pt]{article}

\usepackage[a4paper, left=3.2cm,top=3cm,right=3.2cm,bottom=3.6cm]{geometry}
\usepackage[T1]{fontenc}
\usepackage[utf8]{inputenc}

\usepackage{filecontents}
\RequirePackage{filecontents}

\usepackage{cite}

\usepackage[utf8]{inputenc}
\usepackage[T1]{fontenc}
\usepackage{lmodern}
\usepackage{epstopdf}
\usepackage[english]{babel}
\usepackage{amsmath,amsthm,amssymb}
\usepackage{dsfont}
\usepackage{mathtools}
\usepackage{subcaption}
\usepackage[color=green!40]{todonotes}
\usepackage{graphicx}
\usepackage{tikz}
\usepackage{bbm}
\usepackage{enumitem}
\usepackage{tikz}

\usepackage[binary-units]{siunitx}
\usepackage{eurosym}

\usepackage{todonotes}

\newcommand*{\N}{\mathds{N}}

\newcommand*{\R}{\mathds{R}}

\newcommand{\Ro}{\mathbf{R}}

\newcommand{\ww}{\mathbf{w}}

\newcommand{\edot}{\, \cdot \, }

\newcommand{\regg}[1]{\lVert#1\rVert_{1,\ww}}

\newcommand{\Ho}{\mathbf H}
\newcommand{\Lo}{\mathbf L}
\newcommand{\Ko}{\mathbf K}
\newcommand{\Uo}{\mathbf U}

\newcommand{\U}{\mathbb U}
\newcommand{\V}{\mathbb V}

\newcommand{\ph}{\varphi}
\newcommand{\Ph}{\boldsymbol \Phi}

\newcommand{\al}{\alpha}
\newcommand{\la}{\lambda}
\newcommand{\La}{\Lambda}
\newcommand{\reg}{\mathcal R}

\newcommand{\encoder}{\mathbf E}
\newcommand{\decoder}{\mathbf D}

\newcommand{\unet}{\mathcal N}
\newcommand{\signal}{u}
\newcommand{\data}{v}
\newcommand{\synthesis}{\mathcal{S}_{\alpha, \data_\delta}}

\newcommand{\xialpha}{\xi_{\alpha, \delta}}

\DeclarePairedDelimiter{\abs}{\lvert}{\rvert}
\DeclarePairedDelimiter{\norm}{\lVert}{\rVert}
\DeclarePairedDelimiter{\innerprod}{\langle}{\rangle}
\DeclarePairedDelimiter{\kl}{(}{)}
\DeclarePairedDelimiter{\set}{\{}{\}}

\DeclareMathOperator{\ran}{ran}

\DeclareMathOperator*{\argmin}{arg\,min}

\DeclareGraphicsExtensions{.eps,.pdf,.png,.jpg}

\newtheorem{theorem}{Theorem}

\newtheorem{assumption}[theorem]{Assumption}
\newtheorem{proposition}[theorem]{Proposition}
\newtheorem{remark}[theorem]{Remark}

\numberwithin{table}{section}
\numberwithin{figure}{section}
\numberwithin{theorem}{section}

\usepackage{authblk}
\usepackage{soul}
\usepackage{xcolor}
\colorlet{lred}{red!40}
\colorlet{lgreen}{green!40}
\colorlet{lblue}{blue!40}
\colorlet{lviolet}{violet!40}

\begin{document}

\author{Daniel Obmann}

\affil{Department of Mathematics\authorcr
University of Innsbruck\authorcr
Technikerstrasse 13, 6020 Innsbruck, Austria\authorcr
 {\tt daniel.obmann@uibk.ac.at}
 }

\author{Johannes Schwab}

\affil{Department of Mathematics, University of Innsbruck\authorcr
Technikerstrasse 13, 6020 Innsbruck, Austria\authorcr
 {\tt Johannes.Schwab@uibk.ac.at}
 }

\author{Markus Haltmeier}

\affil{Department of Mathematics, University of Innsbruck\authorcr
Technikerstrasse 13, 6020 Innsbruck, Austria\authorcr
 {\tt markus.haltmeier@uibk.ac.at}
 }

\markboth{Preprint submitted to IEEE Trans. Image Process.}%
{D. Obmann, J. Schwab and M. Haltmeier}

\title{Deep synthesis regularization of inverse problems}
\maketitle

\begin{abstract}
Recently, a large number of efficient deep learning methods for solving inverse problems have been developed and show outstanding numerical performance.  For these deep learning methods, however, a solid theoretical foundation in the form of reconstruction guarantees is missing. In contrast, for classical reconstruction methods, such as convex variational and frame-based regularization, theoretical convergence and convergence rate results are well established.  In this paper, we introduce deep synthesis regularization (DESYRE) using neural networks  as nonlinear synthesis operator bridging the gap between these two worlds. The proposed method allows to exploit the deep learning benefits of being well adjustable to available training data and on the other hand comes with a solid mathematical foundation.  We present a complete convergence analysis with convergence rates for the proposed deep synthesis regularization. We present a strategy for constructing a synthesis network as part of an analysis-synthesis sequence together with an  appropriate training strategy.  Numerical results show the plausibility of our approach.
\end{abstract}

\section{Introduction} \label{sec:intro}

Inverse problems naturally arise  in a wide range of important imaging applications, ranging from computed tomography, remote sensing to image restoration. Such application can be formulated as the task of reconstructing the unknown
 image  $\signal    \in \U$ from data
\begin{equation} \label{intro:ip}
\data_\delta  = \Ko \signal  + z_\delta  \,.
\end{equation}
Here $\Ko \colon \U \rightarrow \V$ is a linear operator between Hilbert spaces and
$z_\delta \in \V$ is the data distortion. Moreover, we assume $\norm{z_\delta}\leq \delta$, where the
index $\delta \geq 0$ denotes noise level. For $\delta =0$ we call
$\data  = \Ko \signal$ the noise free equation.

\subsection{Regularization}
\label{sec:reg}

Inverse problems are typically ill-posed. This means that even in the noise free case the solution of \eqref{intro:ip} is not unique or is unstable with respect to data perturbations. In order to overcome the ill-posedness, 
regularization methods have to be applied which incorporate suitable  prior information that acts as a selection
criterium and at the same  time stabilizes the reconstruction  \cite{scherzer2009variational,EngHanNeu96}.

One of the most established stable  reconstruction approaches
is convex variational regularization. In this case one considers minimizers of
the generalized Tikhonov functional 
\begin{equation} \label{eq:intro:var}
 \norm{\Ko \signal - \data_\delta}^2 + \alpha \reg_{\U}(\signal) \to  \min_{\signal}  \,,
\end{equation}
where $\reg_{\U} \colon \U \rightarrow [0, \infty]$ is a convex regularizer  
on the signal space $\U$ and $\alpha > 0$ is the regularization parameter.
Several choices for the regularizer have been proposed and  analysed.
For example, the choice $\reg_{\U}(\signal) = \norm{\signal}_\U^2$ leads to  quadratic Tikhonov regularization, 
choosing the regularizer as the total variation (TV) semi-norm $\reg_\U(\signal) = \norm{\nabla \signal}_1$ yields to TV-regularization, and  choosing  the $\ell^1$-norm $\reg_{\U}(\signal) = \sum_{ i \in \N} \abs{\innerprod{e_i, \signal}}$ with respect to a given orthonormal basis $(e_i)_{i \in \N}$ yields to sparse 
$\ell^1$-regularization. 

Convex variational regularization is build on a solid theoretical fundament. 
In particular, if $\reg_\U$ is convex, lower semi-continuous and coercive on $\ker(\Ko)^\bot$, then  
\eqref{eq:intro:var} is well-posed, stable  and convergent as $\alpha \to 0$. 
Moreover, for elements $\signal \in \U$ satisfying the so-called source condition 
$\ran(\Ko^* )\cap \partial \reg_{\U} (\signal) \neq \emptyset$, convergence rates in the form of quantitative estimates between $\signal$ and solutions  of \eqref{eq:intro:var}  have been derived~\cite{scherzer2009variational,burger2004convergence}.

\subsection{Frame-based methods} 
\label{sec:frames}

Variational regularization \eqref{eq:intro:var} is based on the  assumption that  a small value of the regularizer is a good prior for the underlying  signal-class. However, it is often challenging to hand-craft an appropriate regularization term for a given class of images. 
Frame-based methods address this issue by adjusting frames   to the signal  class and  using a small value of the regularizer of the frame coefficients 
as image prior.   
  
Let $(\ph_\la)_{\la  \in \La}$ be a frame of $\U$ and denote by
$\Ph \colon  \ell^2(\La) \to  \U$  the  synthesis operator 
that maps  $\xi \in \ell^2(\La)$
to the synthesized signal $\Ph(\xi) \coloneqq \sum_{\la \in \La} \xi_\la \ph_\la$.
Its adjoint  $\Ph^*$ is  the  analysis operator and maps $\signal \in \U$ to the so-called analysis coefficients 
$\Ph^*(\signal) = (\innerprod{\ph_\la, \signal})_{\la \in \La}$.
Two established frame based approaches are the following 
 frame synthesis and frame analysis regularization, respectively,  
\begin{align}
	\label{eq:syn}
	\xi_{\al,\delta}^{\rm syn} 
	&\in  	
	\argmin_{\xi}  \Bigl\{  
	\norm{\Ko \Ph (\xi)  - \data_\delta}^2 
	+ \alpha \reg (\xi) \Bigr\}
	 \\
	  \label{eq:ana}
	u_{\al,\delta}^{\rm ana} 
	&\in  
	\argmin_{\signal} 
	\Bigl\{ 
	\norm{\Ko \signal - \data_\delta}^2 + \alpha \reg(\Ph^*(\signal)) \Bigr\} 
	\,.
\end{align}
Here  $\reg \colon \ell^2(\La) \rightarrow [0, \infty]$ is a convex regularizer  on the  coefficient space  and $\alpha > 0$ the  regularization parameter.
Typical instances of frame analysis and frame synthesis regularization are when the regularizer $\reg = \regg{\edot}$ is taken as the (weighted) $\ell^1$-norm given by
\begin{equation}  \label{eq:snett1}
\regg{\xi} \coloneqq \sum_{ \la \in \La} w_\la \abs{\xi_\la} \,,
\end{equation}
where $w_\la > 0$.  In this case  \eqref{eq:ana} enforces sparsity of the  analysis coefficients  $\Ph^*(\signal)$, whereas \eqref{eq:syn} enforces sparsity of the    synthesis coefficients of the signal. In the case of bases, the  two approaches \eqref{eq:syn} and \eqref{eq:ana} are equivalent (if  one of them is applied with the dual basis). In the redundant case, however, they are fundamentally different 
\cite{elad2007analysis,frikel2019sparse}.

Many different frame-based approaches for solving inverse problems  have been analyzed \cite{dong2013x, haltmeier2013stable, chaux2007variational,elad2007analysis,frikel2019sparse}. However, these methods rely on linear synthesis operators which may not be appropriate for 
a given signal-class.  Recently, non-linear  deep learning and neural network based approaches showed  outstanding performance for various imaging  applications. Inspired by such methods, in this paper, we generalize the 
 frame based synthesis approach to allow neural networks as synthesis operators. Because these representations  are non-linear, the resulting approach requires new mathematical theory that we develop in this paper.

%
%
%

\subsection{Deep synthesis regularization}
\label{sec:reg}

In  this paper,  we propose deep synthesis regularization  (DESYRE)
where we consider minimizers of  
\begin{equation} \label{eq:snett}
\synthesis (\xi) \coloneqq   
\norm{(\Ko \circ \decoder_{\al}) (\xi) - \data_\delta}^2 
+ \alpha \regg{\xi} \,.
\end{equation}
Here, $\decoder_{\al} \colon \ell^2(\La) \rightarrow \U$ are  possibly non-linear synthesis  mappings  and $\regg{\edot} \colon \ell^2(\La) \rightarrow [0, \infty]$ is  the weighted $\ell^1$-norm defined by \eqref{eq:snett1}.
The main theoretical results  of this paper provide a complete convergence analysis for the DESYRE approach.
The  corresponding proofs  closely follow  \cite{grasmair2008sparse,scherzer2009variational}. 
We point out, however, that the convergence  results in  these references  cannot be directly 
applied  to our setting because we allow the coefficient operators  $\Ko \circ \decoder_{\al}$ in \eqref{eq:snett} to depend on  $\alpha$. Using  non-stationary synthesis mappings $\decoder_{\al}$ allows accounting for  discretization as  well as for approximate network training.  
In the recent years  various deep learning based reconstruction  methods have been derived which  outperform classical variational regularization \cite{han2018framing, lee2017deep, adler2017solving, jin2017deep}. However, rigorously  analyzing them as regularization methods  is challenging. DESYRE  follows the deep learning strategy and, as we demonstrate in this paper, allows the derivation of results similar to classical sparse regularization.  

In \cite{li2018nett}, a somehow dual approach  to  \eqref{eq:snett} has been studied, 
where minimizers of the NETT functional 
$ \norm{\Ko \signal - \data_\delta}^2 + \alpha \reg(\encoder(\signal))$
have been considered, where $\encoder \colon \U \rightarrow \ell^2(\La)$ is a non-linear analysis operator and $\reg{}$ a  regularizer. Due to the  non-linearity  of $\encoder$, the penalty  $\reg \circ \encoder$ is typically non-convex. One advantage  of the deep synthesis method over the analysis counterpart is  that the penalty term is  still convex,  which is beneficial for the theoretical analysis as well as the numerical minimization.
Another strength of DESYRE is that the network can be trained without explicit knowledge of the operator $\Ko$. Thus, the proposed approach has some kind of universality  like \cite{rick2017one, li2018nett, romano2017little, lunz2018adversarial, aggarwal2018modl} in the sense  that the  network is trained independent of the specific forward operator and can be used for different inverse problems without retraining.

\subsection{Outline}

The rest of the paper is organized as follows. 
Section~\ref{sec:conv} gives a theoretical analysis of the proposed method.
In Section~\ref{sec:trainingstrategy} we propose a learned synthesis mapping. We present numerical results and compare DESYRE to other reconstruction methods in Section~\ref{sec:num}. The paper ends with a summary and outlook presented in Section \ref{sec:outlook}. 
This paper is a significantly changed and extended version of the proceedings \cite{obmann2019sparse} presented at the SampTA 2019 in Bordeaux. The analysis of the proposed method and all the numerical results are completely new.

\section{Convergence analysis}
\label{sec:conv}

This section gives a complete convergence analysis of DESYRE together with  convergence rates. 
For the following analysis  consider  the coefficient equation
\begin{equation} \label{eq:sip}
	\data  =  (\Ko \circ \decoder) (\xi)   \,,
\end{equation}
where  $\Ko$ is the linear forward operator and  
$\decoder$ a possibly non-linear synthesis operator. 
In the case of noisy data, we approach \eqref{eq:sip} by deep synthesis regularization \eqref{eq:snett}, \eqref{eq:snett1} with variable synthesis operators $(\decoder_{\al})_{\al >0}$. Whenever it is clear from the context which norm is used, we omit the subscripts.

\subsection{Well-posedness}

Throughout this paper we assume that the following assumption holds. For the following let $\decoder, \decoder_{\al} \colon  \ell^2(\La) \to \U$ for $\al>0$ be given. 
    
\begin{assumption}[Deep synthesis regularization] \hfill \label{ass:well}
\begin{enumerate}[label=(R\arabic*), leftmargin=3em]
\item \label{ass:well1} $\U$ and $\V$  are Hilbert spaces;
\item \label{ass:well2} $\Ko \colon \U \rightarrow \V$ is linear and bounded;
\item \label{ass:well3} $\La$ is an at most countable set;
\item \label{ass:well5} $\forall \al > 0 \colon \decoder_{\al}$ is  weakly sequentially continuous;
\item \label{ass:well6} $(w_\la)_{\la \in \La} \in (0, \infty)^\La$  satisfies $\underline w \coloneqq \inf_{\la \in \La} w_\la > 0$.
\end{enumerate}
\end{assumption}

Assumption \ref{ass:well6}  implies that $\regg{\edot}$ is coercive. To see this, let  
$\xi \in \ell^2(\La)$ and assume  $\regg{\xi}  < \infty$. 
Then,   for all $\la \in \La$, we have $w_\la \abs{\xi_\la} / \regg{\xi} \leq 1$  and hence  
\begin{equation*}
\frac{\underline w^2}{\regg{\xi}^2}  \norm{\xi}_2^2 = 
\sum_{ \la \in \La}  \left(  \frac{\underline w \, \abs{ \xi_\la}}{\regg{\xi}}\right)^2 \leq \sum_{ \la \in \La} \left( \frac{w_\la \abs{ \xi_\la}}{\regg{\xi}} \right)^2 
\leq \sum_{ \la \in \La} \frac{ w_\la \abs{\xi_\la}}{\regg{\xi}} = 1.
\end{equation*}
This yields the estimate  
$\norm{\xi}_2 \leq \regg{\xi} / \underline w$ and proves the coercivity of $\regg{\edot}$. Since the data-discrepancy term is non-negative this also shows the coercivity of the synthesis functional $\synthesis$ for every $\alpha > 0$.

\begin{remark} The above proof relies on the fact that $x \geq x^2$ for  $x \in [0,1]$. Since the inequality $x^q \geq x^2$ on $[0,1]$  holds for any $q \leq 2$, the above proof can be done with a weighted $\ell^q$-norm.  Similarly, the  following well-posedness and the convergence results also hold  for   the (weighted) $\ell^q$-regularizer  
$\norm{\xi}_{q, \ww} = \sum_{ \la \in \La} w_\la \abs{\xi_\la}^q$ with  $q \in  [1,2]$. 
\end{remark}

As the sum of non-negative convex and weakly continuous functionals,
$\regg{\edot}$ is convex and weakly lower semi-continuous.  
Assumptions  \ref{ass:well2}, \ref{ass:well5} imply that $(\Ko \circ \decoder_{\al})$ is weakly sequentially continuous. Moreover, the norm $\norm{\,\cdot\,}_\V$ is weakly sequentially lower semi-continuous. This shows that $\synthesis$  is weakly 
lower semi-continuous as a sum of weakly lower semi-continuous functionals. 
Coercivity and weak  lower semi-continuity basically yield the following 
well-posedness results for deep synthesis regularization.

\begin{theorem}[Well-posedness] \label{thm:wellposed}
Let Assumption \ref{ass:well} be satisfied, 
$\data_\delta \in \V$ and $\alpha > 0$. 
Then the following hold:
\begin{enumerate}[label=(\alph*)] 
\item Existence: $\synthesis$ has at least one minimizer.
\item Stability: Let $(\data_k)_{k\in \N} \in \V^\N$ satisfy $\data_k \rightarrow \data_\delta$ and choose $\xi_k \in \argmin \mathcal{S}_{\alpha, \data_k}$. 
Then $(\xi_k)_{k\in \N}$ has a convergent subsequence and the 
limit of every convergent subsequence is a minimizer of $\synthesis$.
\end{enumerate}
\end{theorem}

\begin{proof}
Assumptions   \ref{ass:well2}, \ref{ass:well5}
imply that $(\Ko \circ \decoder_{\al})$ is weakly sequentially continuous.
Therefore the results follow from \cite[Propositions 5 and 6]{grasmair2008sparse}. 
\end{proof}

\subsection{Convergence}


An  element in the set $\argmin \set{\regg{\xi} \mid  (\Ko \circ \decoder) (\xi) = \data}$ is called  $\regg{\edot}$-minimizing solution of  \eqref{eq:sip}.  
Note that $\regg{\edot}$-minimizing solutions exists whenever there is any solution $\xi$ with $\regg{\xi} < \infty$, see \cite[Theorem 3.25]{scherzer2009variational}.
For  the following convergence results we
make the following additional assumption
\begin{enumerate}[label=(R\arabic*), leftmargin=3em] \setcounter{enumi}{5}

\item \label{ass:well7}
$\forall \rho > 0 \colon \Delta_\rho(\al) \coloneqq \sup_{\norm{\xi} \leq \rho}\norm{\decoder_\al (\xi) - \decoder (\xi)} \to 0$.     
\end{enumerate}
This assumption guarantees that $\decoder_\al$ arbitrarily well  approximates $\decoder$ as $\al \to 0$.

\begin{theorem}[Convergence] \label{thm:conv}
Let Assumption \ref{ass:well}  and \ref{ass:well7} hold, let $\data \in \V$,  $\xi^\dag$ be an $\regg{\edot}$-minimizing 
solution of \eqref{eq:sip} and choose  $\alpha \colon  (0, \infty ) \to (0, \infty )$ 
such that
\begin{equation} \label{eq:alpha}
\lim_{\delta \to 0} \alpha(\delta) 
= 
\lim_{\delta \to 0}\frac{\Delta_\rho (\alpha(\delta))^2  }{\alpha(\delta)} 
= 
\lim_{\delta \to 0}\frac{\delta^2}{\alpha(\delta)} 
= 0
\,.
\end{equation}
Moreover,  let $(\delta_k)_{k \in \N} \in (0,\infty)^\N$, $(\data_k)_{k \in \N} \in \V^\N$ satisfy  $\delta_k \rightarrow 0$, $\norm{y - y_k} \leq \delta_k$ and choose $\xi_k \in \argmin \mathcal{S}_{\alpha(\delta_k), \data_k}$.
Then, the following hold:    
\begin{enumerate}[label=(\alph*)] 
\item\label{thm:conv1} $(\xi_k)_{k\in \N}$ has a convergent subsequence.
\item\label{thm:conv2} The limit of every  convergent subsequence $(\xi_{k(n)})_{n \in \N}$ 
is an  
$\regg{\edot}$-minimizing solution of \eqref{eq:sip}.
\item\label{thm:conv3}  If the $\regg{\edot}$-minimizing solution of the coefficient equation \eqref{eq:sip} is unique, then 
$\xi_k \to \xi^\dag$.
\end{enumerate}
\end{theorem}

\begin{proof} 
Let $\rho > \norm{\xi^\dagger}$ and write $\al_k \coloneqq \alpha(\delta_k)$,
$\decoder_k \coloneqq \decoder_{\al_k}$. By definition of $\xi_k$, we have
\begin{align}\nonumber
	\mathcal{S}_{\alpha_k, \data_k}(\xi_k)
	&= 
	\norm{(\Ko \circ \decoder_{k}) (\xi_k) - \data_k}^2 
	+ \alpha_k \regg{\xi_k}
	\\ \nonumber
	 & 
	\leq
	  \norm{(\Ko \circ \decoder_k) (\xi^\dag) - \data_k}^2 
+ \alpha_k \regg{\xi^\dag}
\\ \label{eq:conv-aux}
	&\leq 
	 \kl{ \norm{\Ko} \Delta_{\rho}(\al_k)  + \delta_k }^2  
	+  \alpha_k \regg{\xi^\dag} \,.
\end{align}
The right hand side in \eqref{eq:conv-aux}  tends to $0$,
which together with the estimate $\norm{\data-\data_k} \leq \delta_k$ and \eqref{eq:alpha} yields 
\begin{align}      \label{eq:conv-aux2}
 \lim_{ k \to \infty} \norm{(\Ko \circ \decoder_k) (\xi_k) - \data} &=  0  \,,
\\ 
 \label{eq:conv-aux3}
 \limsup_{k \rightarrow \infty} \regg{\xi_k} &\leq \regg{\xi^\dag} \,.
 \end{align}
The coercivity of $\regg{\edot}$ and \eqref{eq:conv-aux3}  in turn imply that there is  some  weakly convergent subsequence  $(\xi_{k(n)})_{n \in \N}$. We denote its  weak limit  by $\hat{\xi}$. 

Using \eqref{eq:conv-aux2} and  \ref{ass:well7}, we see that $\hat{\xi}$ solves \eqref{eq:sip}.
The lower semi-continuity of $\regg{\edot}$ and \eqref{eq:conv-aux3} imply  
\begin{equation*}
\regg{\hat{\xi}} \leq \liminf_{n \rightarrow \infty} \regg{\xi_{k(n)}} \leq \limsup_{n \rightarrow \infty} \regg{\xi_{k(n)}} \leq \regg{\xi^\dag}  \,.
\end{equation*}
Hence  $\hat{\xi}$ is an $\regg{\edot}$-minimizing solution of \eqref{eq:sip} and 
$\regg{\xi_{k(n)}} \to  \regg{\xi^\dag}$.
According to \cite[Lemma~2]{grasmair2008sparse}, weak convergence of $(\xi_{k(n)})_{n \in \N}$ together with the convergence of $\regg{\xi_{k(n)}}$ implies 
$\norm{\xi_{k(n)} - \hat \xi} \to 0$.  
Finally, if  \eqref{eq:sip} has a unique $\regg{\edot}$-minimizing solution, 
then every  subsequence of $(\xi_{k})_{k \in \N}$ has a subsequence converging   
to $\xi^\dag$, which implies  
 $\xi_k \to  \xi^\dag$.
\end{proof}

The existence of a solution $\xi^\dag$ to \eqref{eq:sip} always implies the existence of a solution  $\signal^\dag \coloneqq \decoder(\xi^\dag)$ to the original problem.
Moreover, Theorem~\ref{thm:conv} shows strong convergence of the  regularized solutions in the coefficient space $ \ell^2(\La)$. By assuming that the synthesis mappings  $\decoder_\al$ are uniformly Lipschitz continuous we further  get the strong convergence of the regularized solutions in the signal space.

\begin{theorem}[Convergence in signal space] \label{thm:conv-x}
Let the assumptions of Theorem~\ref{thm:conv} hold. Assume, additionally, that
$(\Ko \circ \decoder)(\xi) = \data$ has a unique $\regg{\edot}$-minimizing solution $\xi^\dag$,
 and that  $(\decoder_{\al})_{\al > 0}$ are uniformly Lipschitz. 
 Consider  $\xi_k \in \argmin \mathcal{S}_{\al, \data_k}$, and define  $\signal^\dag \coloneqq \decoder (\xi^\dag)$, $\signal_k \coloneqq \decoder_{\al(\delta_k)} (\xi_k)$.
Then $\lim_{k \to \infty} \norm{\signal_k - \signal^\dag} \rightarrow 0$.

\end{theorem}

\begin{proof}
Following the notions of the proof of Theorem~\ref{thm:conv}  we have 
\begin{align*}
\norm{\signal^\dag - \signal_k} 
&= \norm{\decoder(\xi^\dag) - \decoder_{k}(\xi_k)} 
\\ 
&\leq  
\norm{\decoder(\xi^\dag) - \decoder_{k}(\xi^\dag)} 
+
\norm{\decoder_{k}(\xi^\dag) - \decoder_{k}(\xi_k)} 
\\
&\leq  
\norm{\decoder(\xi^\dag) - \decoder_{k}(\xi^\dag)} 
+
L \norm{\xi^\dag - \xi_k} \,, 
\end{align*}
where $L$ is a uniform Lipschitz-constant for $(\decoder_{\al})_{\al > 0}$.
By assumption, $\norm{\decoder(\xi^\dag) - \decoder_k(\xi^\dag)} \to 0$ and according   to Theorem~\ref{thm:conv} we have $\norm{\xi^\dag - \xi_k}
\to 0$. This  yields $\norm{\signal^\dag - \signal_k} \to 0$ and concludes the proof.  
\end{proof}


\begin{remark}
With $\xi_k \in \argmin \mathcal{S}_{\alpha, \data_k}$,  Theorem~\ref{thm:conv}  states that  $(\xi_k)_{k\in \N}$ has a subsequence $(\xi_{k(n)})_{n\in \N}$ converging   to some $\regg{\edot}$-minimizing solution  $\xi^\dag$
of $(\Ko \circ \Ph)(\xi) = \data$. If we additionally assume that $(\decoder_{\al})_{\al >0}$ are uniformly  Lipschitz-continuous, then following the proof of 
Theorem~\ref{thm:conv-x}  one shows that   $\decoder_{k(n)}(\xi_{k(n)}) \to  \decoder(\xi^\dag) = \signal^\dag$.
In particular,  the limits $\signal^\dag$ are characterized  as solutions of $\Ko \signal = \data$ having  a representation using an $\regg{\edot}$-minimizing solution $\xi^\dagger$ of the coefficient problem $(\Ko \circ \Ph)(\xi) = \data$.
\end{remark}

A network is typically given as a composition of trained Lipschitz-continuous functions and given activation functions. The standard activation functions all satisfy a Lipschitz-continuity condition, e.g. the ReLU, $\tanh$ or sigmoid function are all Lipschitz-continuous.  The uniform Lipschitz-continuity assumption on the networks $\decoder_{\al}$ can therefore  easily  be fulfilled.


\subsection{Convergence rates}


Convergence rates name quantitative error estimates  between exact and regularized solutions. In order to derive such results we have to make some additional assumptions on the interplay between the regularization functional and the operators  $\Ko$, $\decoder$, $\decoder_\al$.


\begin{assumption} \label{ass:wellrate}
Let  $\xi^\dag$ be an 
$\regg{\edot}$-minimizing  solution of  \eqref{eq:sip}
with $\data \in \V$. Assume there are  
$\beta_1, \beta_2, \sigma > 0$ and
$\rho > \regg{\xi^\dag}$ such that
\begin{equation*}
\regg{\xi} - \regg{\xi^\dag} \geq \beta_1 \norm{\xi - \xi^\dag} 
- \beta_2 \norm{\Ko \kl{  \decoder(\xi) -  \decoder(\xi^\dag) }}
\end{equation*}
for all $\xi$ with $\regg{\xi} < \rho$ and 
$ \norm{\Ko \kl{  \decoder(\xi) -  \decoder(\xi^\dag) }} < \sigma$.
\end{assumption}


\begin{proposition}[Quantitative error estimates] \label{prop:rate}
Let Assumptions \ref{ass:well}, \ref{ass:wellrate} and  \ref{ass:well7} hold. Furthermore, assume that $\data_\delta \in \V$ satisfies $\norm{\data_\delta - \data} \leq \delta$ choose  $\xi_{\alpha, \delta} \in \argmin \synthesis$, and set $h_{\al, \delta} \coloneqq  \norm{\Ko \kl{ \decoder_{\al} (\xialpha) -  \decoder (\xialpha)}}$ and $g_{\rho, \al, \delta} \coloneqq (\delta + \norm{\Ko}\Delta_\rho(\al))$.
Then, the  following hold:
\begin{align*}
&\norm{\xi_{\alpha, \delta} - \xi^\dag} \leq \frac{g_{\rho, \al, \delta}^2 + \alpha \beta_2 (\delta+h_{\al, \delta})  + \frac{(\alpha \beta_2)^2}{2} }{\alpha \beta_1} \,,\\
&\norm{(\Ko \circ \decoder_{\al}) (\xi_{\alpha, \delta}) - \data_\delta}^2 \leq 2g_{\rho, \al, \delta}^2 
+ 2\alpha \beta_2 (\delta+h_{\al, \delta})  + (\alpha \beta_2)^2\,.
\end{align*}
\end{proposition}

\begin{proof}
By definition of $\xi_{\alpha, \delta}$ we have 
\begin{align*}
\synthesis(\xialpha) &\leq \synthesis(\xi^\dag)\\
&=   \norm{(\Ko \circ \decoder_{\al}) (\xi^\dag) - \data_\delta}^2 + \alpha \regg{\xi^\dag} \\
&\leq   g_{\rho, \al, \delta}^2 + \alpha \regg{\xi^\dag} \,.
\end{align*}
Using Assumption~\ref{ass:wellrate} we get
\begin{align*}
g_{\rho, \al, \delta}^2 
&\geq \synthesis(\xialpha) - \alpha \regg{\xi^\dag} 
\\
&=\norm{(\Ko \circ \decoder_{\al}) (\xialpha) - \data_\delta}^2
+ \alpha \kl{ \regg {\xialpha}
-  \regg{\xi^\dag} } 
\\
&\geq \norm{(\Ko \circ \decoder_{\al}) (\xialpha) - \data_\delta}^2 
+ 
\alpha \beta_1 \norm{\xialpha - \xi^\dag} 
\\
&\hspace{3em} - \alpha \beta_2 
\norm{\Ko \kl{  \decoder(\xialpha) -  \decoder (\xi^\dag) }}
\\
&\geq  \norm{(\Ko \circ \decoder_{\al}) (\xialpha) - \data_\delta}^2 +
\alpha \beta_1 \norm{\xialpha - \xi^\dag}  \\
&\hspace{3em}  -\alpha \beta_2 \norm{(\Ko \circ \decoder_{\al}) (\xialpha) - \data_\delta}  - \alpha \beta_2 \delta\\
&\hspace{3em} 
- \alpha \beta_2 \norm{\Ko \kl{ \decoder_{\al} (\xialpha) -  \decoder (\xialpha)}}  \,.
\end{align*}
Applying Young's inequality $ab \leq a^2/2+b^2/2$  with  
$a = \alpha \beta_2$ and 
$b = \norm{(\Ko \circ \decoder_{\al}) (\xialpha) - \data_\delta}$ shows 
\begin{multline*}
g_{\rho, \al, \delta}^2 + \frac{(\alpha \beta_2)^2}{2} + \alpha \beta_2 (\delta+h_{\al, \delta}) \geq \frac{1}{2} \norm{(\Ko \circ \decoder_{\al}) (\xialpha) - \data_\delta}^2 + \alpha \beta_1 \norm{\xialpha - \xi^\dag}
\end{multline*}
and the assertion follows since the terms on the right hand side are non-negative.
\end{proof}

In the following we write $\alpha \asymp \delta$  if 
 $C_1 \delta \leq \alpha(\delta) \leq C_2 \delta$ for  $\delta >0$ and some constants $C_1, C_2 > 0$.


\begin{theorem}[Linear convergence rate] \label{thm:rate}
Let Assumptions~\ref{ass:well}, \ref{ass:wellrate} and  \ref{ass:well7} hold, let  $\alpha = \alpha(\delta)$ be such that $\alpha, h_{\alpha, \delta}, \Delta_\rho(\alpha(\delta)) \asymp \delta$ and choose $\xi_{\alpha, \delta} \in \argmin \synthesis$. 
Then the following hold:
\begin{enumerate}[label=(\alph*)]
\item $\norm{\xialpha - \xi^\dag} = \mathcal{O}(\delta)$,
\item Assume   that $\decoder_{\al}$ is Lipschitz and set $\signal_{\alpha, \delta} \coloneqq \decoder_{\al} (\xialpha)$, $\signal^\dag \coloneqq \decoder (\xi^\dag)$. Then 
$\norm{\signal_{\alpha, \delta} - \signal^\dag} = \mathcal{O}(\delta)$.
\end{enumerate}
\end{theorem}

\begin{proof}
This is an immediate consequence of the first inequality in Proposition~\ref{prop:rate}.   
\end{proof}


%

\begin{remark} \label{rem:ass-rate}
The convergence rate condition in Assumptions \ref{ass:well} is    
the same as by taking $q=r=1$ and choosing $\Ko \circ \decoder$ for the forward operator in 
\cite[Assumption 1]{grasmair2008sparse}. As shown in \cite{grasmair2008sparse}, this assumption is satisfied  if $ \Ko \circ  \decoder$ is  linear and injective on certain finite dimensional subspaces, and the so-called source condition $\partial  \regg{\edot} (\xi^\dag ) \cap \ran\kl{ (\Ko \circ  \decoder)^*} \neq \emptyset$ is satisfied. The latter condition is in particular satisfied if  $\La$ is finite and $\Ko \circ \decoder$ is injective.
\end{remark}

\section{Learned synthesis operator} \label{sec:trainingstrategy}

In this section we propose  a non-linear learned synthesis operator  which is  part of a sparse encoder-decoder pair. We describe  a modified U-net  that we use for the network architecture and  give details on the network training.

\subsection{Sparse encoder-decoder pair} 

Following the deep learning paradigm we select the synthesis operator 
$\decoder$ from a  parametrized class  $(\decoder_\theta)_{\theta \in \Theta}$ of neural network functions $ \decoder_\theta \colon \ell^2(\La) \to \U$ by adjusting them to certain training images $\signal_1, \dots,  \signal_m \in \U$. Here the subscript $\theta \in \Theta$ refers to the parameters 
of the network  taken from a  finite dimensional Hilbert space.
The  theoretical analysis presented in the previous section is based on the assumption that the images of interest can be represented as $\signal = \decoder_\theta ( \xi )$ where  
$\xi \in \ell^2(\La) $ has  small value of the regularizer $\regg{\edot}$. 
To find a suitable synthesis network $ \decoder_\theta$, we write the  
corresponding coefficients in the form $\xi =  \encoder_\eta  (\signal )$ with another neural network $\encoder_\eta  \colon \U \to \ell^2(\La)$ applied to $\signal$.

In order to enforce sparsity, a reasonable training strategy  is to take the network parameters as solutions of the constraint minimization problem
\begin{equation} \label{eq:constrainedmin} 
\left.\begin{aligned}
&\min_{\theta, \eta}   && \frac{1}{m} \sum_{i=1}^m  \regg{\encoder_\eta(\signal_i)} +  \beta_0 \norm{\theta}^2 + \gamma_0\norm{\eta}^2 \\
&\,  \text{s.t.}  &&\forall i\in \set{1 \dots , m  }\colon  \signal_i = \decoder_\theta (\encoder_\eta(\signal_i))    \,.
\end{aligned}
\right\}\end{equation}
Here  $\beta_0 \norm{\theta}^2 + \gamma_0 \norm{\eta}^2$  is a regularization term 
with   regularization parameters  $\beta_0, \gamma_0 > 0$.
As described in the  following subsection~\ref{sec:details},  
we use a modified tight-frame U-net as actual architecture   for $\decoder_\theta \circ \encoder_\eta$. The weights in   $\regg{\edot}$ are chosen as  $w_\la = 2^{-\ell}$, where $\la = (\ell,j)$ corresponds to the  coefficients in the  $\ell$-th downsampling step  of the network architecture (see Figure~\ref{fig:architecture}).

\subsection{Modified tight frame U-net}
\label{sec:details}

For the numerical results we take $\decoder_\theta \circ \encoder_\eta$ as a variation  of the tight-frame U-net  \cite{han2018framing}, where  we do not include the bypass-connections from the sequential layer to the concatenation layer.  The input is passed to a sequential layer which consists of convolution, batch normalization and a ReLU activation. The first sequential layer starts with $64$ channels. We pass the output of the first sequential layer to another sequential layer where we reduce the number of channels to $2$ to keep the output dimension of the encoder $\encoder$ reasonable. The output of the second sequential layer is then passed to the downsampling step with the filters $\Ho_\ell$ and   $\Lo_\ell$, that are given by the Haar wavelets low-pass and high-pass filters, respectively. 
The output of the high pass filtered image $\Ho_\ell$ then serves as one set of inputs for the decoder. This downsampling step is then recursively applied to the low  frequency  output $\Lo_\ell$. The low-pass output $\Lo_L$ of the last step of the network is also used as input for the decoder $\decoder$. In each downsampling step the dimensions are reduced by a factor of $2$ while the number of channels is increased by a factor of $2$.  

The upsampling is performed with the transposed filters of the downsampling step. The outputs of the upsampling step $\Ho_\ell^\intercal$ and $\Lo_\ell^\intercal$ are then concatenated and two sequential layers are applied. The channel sizes of these sequential layers are the same as the corresponding first sequential layer of the downsampling step. To obtain the final output we apply a $(1 \times 1)$-convolution with no activation function.
We use $L = 4$ downsampling and upsampling steps.
A visualization of one downsampling and upsampling step is depicted in Figure~\ref{fig:architecture}.

\begin{figure}[h]
\begin{center}
\resizebox{10cm}{!}{
\begin{tikzpicture}
\node (INP) at (0,0)  {};

\node (CONV) at (2,0) [rectangle, minimum height=1.5cm,minimum width=0.25cm, draw=black] {};

\node (BNORM) at (4,0) [rectangle, minimum height=1.5cm,minimum width=0.25cm, draw=black] {};

\draw[->, black, very thick] (INP) to node[above] {\Large Input} (CONV);
\draw[->, black, very thick] (CONV) to node[above] {\Large Seq.} (BNORM);

\node (DH) at (5,-1) [rectangle, minimum height=0.75cm,minimum width=0.5cm, draw=black] {\Large $\Ho_\ell$};
\node (DL) at (5,-2) [rectangle, minimum height=0.75cm,minimum width=0.5cm, draw=black] {\Large $\Lo_\ell$};

\draw[->, blue, very thick] (BNORM.south) |- node[left, circle, draw] {\Large $\downarrow 2$} (DL.west);
\draw[->, blue, very thick] (BNORM.south) |- (DH.west);

\node (REC) at (7,-2) [rectangle, minimum height=1cm,minimum width=1cm, fill=gray!20] {\Large $\unet_{\ell+1}$};

\draw[->, black, thick] (DL) to (REC);

\node (UH) at (9,-1) [rectangle, minimum height=0.75cm,minimum width=0.5cm, draw=black] {\Large $\Ho_\ell^\intercal$};
\node (UL) at (9,-2) [rectangle, minimum height=0.75cm,minimum width=0.5cm, draw=black] {\Large $\Lo_\ell^\intercal$};

\draw[->, black, thick] (REC) to (UL);
\draw[->, black, thick, dashed] (DH) to node[above] {\Large $\decoder$ \Large input} (UH); 

\node (UP1) at (11,0) [rectangle, minimum height=1.5cm,minimum width=0.25cm, draw=black] {};
\node (UP2) at (12,0) {\Large $\overset{\text{\Large Concat.}}{\cdots}$};
\node (UP3) at (13,0) [rectangle, minimum height=1.5cm,minimum width=0.25cm, draw=black] {};

\node (UPDECONV) at (14.5,0) [rectangle, minimum height=1.5cm,minimum width=0.25cm, draw=black] {};

\draw[->, red, very thick] (UL.east) -| node[right, circle, draw] {\Large $\uparrow 2$} (UP1.south);
\draw[->, red, very thick] (UH.east) -| (UP1.south);

\node (HELP) at (7,1) {};

\draw[->, black, very thick] (UP3) to node[above] {\Large Seq.} (UPDECONV);

\end{tikzpicture}
}
\end{center}
\caption{\textbf{Illustration of the used synthesis network.} We start by applying sequential layers which consist of convolution, batch normalization   and activation. After that we use a fixed filter to downsample the features. The network is then recursively applied to the output $\Lo_\ell$. We continue by upsampling using the same filters as before and concatenating the features. Lastly we apply sequential layers again to obtain the output.}
\label{fig:architecture}
\end{figure}
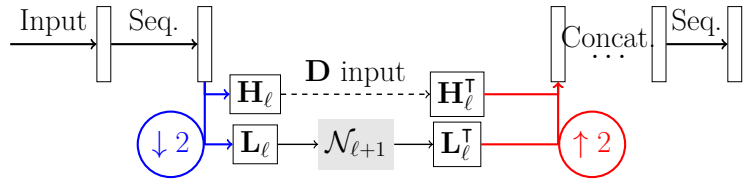 

Table~\ref{tab:network_size} shows the dimension of the feature outputs and the channel sizes for the $\ell^{\text{th}}$ step. Note that the number of channels for the highpass and lowpass filtered output $\Ho_\ell, \Lo_\ell$ is multiplied by $4$ because we have $3$ different highpass filters.

\begin{table}[htb!]
\begin{center}
\begin{tabular}{lrr}
\hline
            &  Dimension &  No. Channels \\
\hline\\
Sequential$_1^\ell$ & $\frac{512}{2^{\ell-1}} \times \frac{512}{2^{\ell-1}}$ & $64 \cdot 2^{\ell-1}$ \\
Sequential$_2^\ell$ & same & $2 \cdot 2^{\ell-1}$ \\
$\Ho_\ell, \Lo_\ell$ & $\frac{512}{2^{\ell}} \times \frac{512}{2^{\ell}}$ & $4 \cdot 2 \cdot 2^{\ell}$ \\

$\Ho^\intercal_\ell, \Lo^\intercal_\ell$ & $\frac{512}{2^{\ell-1}} \times \frac{512}{2^{\ell-1}}$  & $4 \cdot 2 \cdot 2^{\ell}$ \\
Concatenation & same & $4 \cdot 2 \cdot 2^{\ell}$\\
Sequential$_{3,4}^\ell$ & same & $64 \cdot 2^{\ell-1}$\\
\hline
\end{tabular}
\end{center}
\caption{\textbf{Dimension and channel sizes for the $\ell^{\text{th}}$ step of the network.} Same means that the same dimension as the one above is used.}
\label{tab:network_size}
\end{table}

 \subsection{Network training}
 
Finding minimizers of  \eqref{eq:constrainedmin} might be  unstable in practice and difficult  to solve. Therefore we consider a relaxed version where the constraint is added as a penalty. Hence we train the networks  by instead considering the following loss-function
\begin{equation} \label{eq:lossfunction}
\frac{1}{m}\sum_{i=1}^m \norm{\signal_i - \decoder_\theta (\encoder_\eta(\signal_i))}^2 + \alpha  \regg{\encoder_\eta (\signal_i)} + \beta\norm{\theta}^2 + \gamma\norm{\eta}^2 \,,
\end{equation}
where  $\alpha > 0$ is the regularization parameter in \eqref{eq:snett} 
and $\beta \triangleq \al \beta_0 $,  $\gamma \triangleq \al \gamma_0$.
In the numerical realization, the parameters have been chosen  empirically as $\al = 10^{-2}$, $\beta = \gamma = 10^{-4}$.
We minimize  \eqref{eq:lossfunction} with the Adam optimizer \cite{kingma2014adam} using proposed hyper-parameter settings for Adam for $150$ epochs and  a batch size of $6$.

As training data we use   the Low Dose CT Grand Challenge dataset provided by Mayo Clinic \cite{mccollough2016tu}. The complete dataset consists of $512 \times 512$ grayscale images from which we selected a subset containing only lung slices. This results in a dataset with a total of $m = 1311$ images of which $1151$ images (corresponding to 8 patients) were used for training purposes and $160$ (corresponding to the remaining 2 patients) were used for testing purposes. Each image was scaled to have pixel-values in the interval $[0,1]$.


\section{Numerical results}
\label{sec:num}

In this section we present numerical results for sparse view CT,  comparing the   DESYRE functional \eqref{eq:snett} with wavelet synthesis regularization, TV-regularization and a post-processing network.

\subsection{Sparse view CT} \label{subsec:comparisons}

For the numerical results we consider the Radon transform \cite{Nat86,helgason1999radon}  with undersampling 
in the angles as a forward operator $\Ko$.
Formally, the Radon transform is given by
\begin{equation*}
\forall (\theta, s) \in [0, \pi) \times \R \colon \quad 
(\Ko u)(\theta, s) = \int_{E(\theta, s)} u(x) \mathrm{d} x\,,
\end{equation*}
where $E(\theta, s) \triangleq \{(x_1, x_2) \in \R^2 \mid 
x_1\cos(\theta) +  x_2\sin(\theta) = s\}$.  The discretization of 
the Radon transform was performed using the Operator Discretization Library \cite{adler2017odl} using $n_\theta = 60$ equidistant angles in $[0, \pi)$ and $n_s = 768$ equidistant parallel beams in the interval  $[-3/2,3/2]$. In the case of noisy data we simulate the noise by adding $5\%$ Gaussian noise to the measurement-data. This data is visualized in Figure~\ref{fig:data}.

\begin{figure}
\begin{center}
\begin{subfigure}{.35\textwidth}
  \centering
  \includegraphics[scale=0.45, trim={3cm 1cm 3cm 1cm},clip, angle=90]{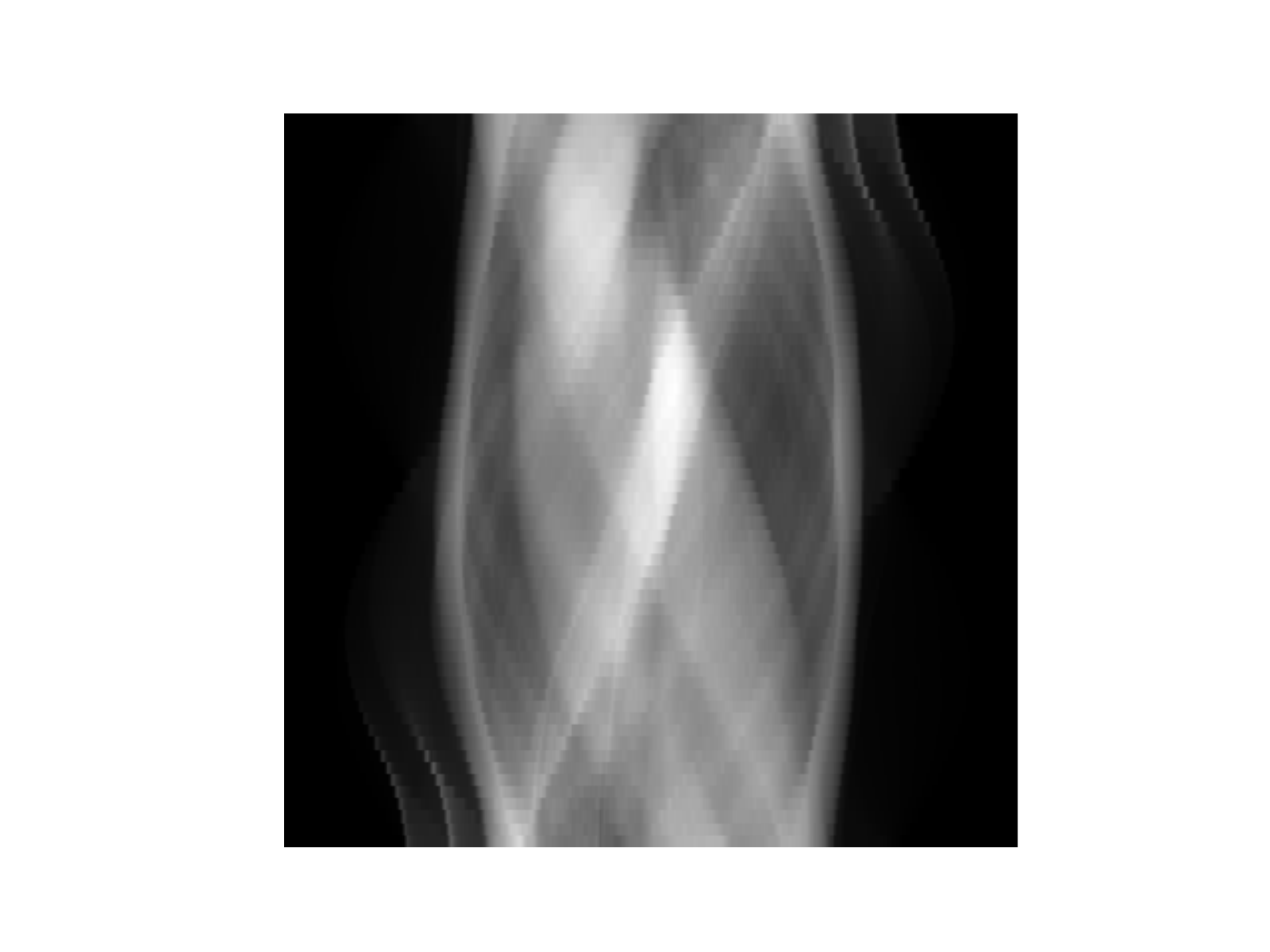}
\end{subfigure}%
\begin{subfigure}{.35\textwidth}
  \centering
  \includegraphics[scale=0.45, trim={3cm 1cm 3cm 1cm},clip, angle=90]{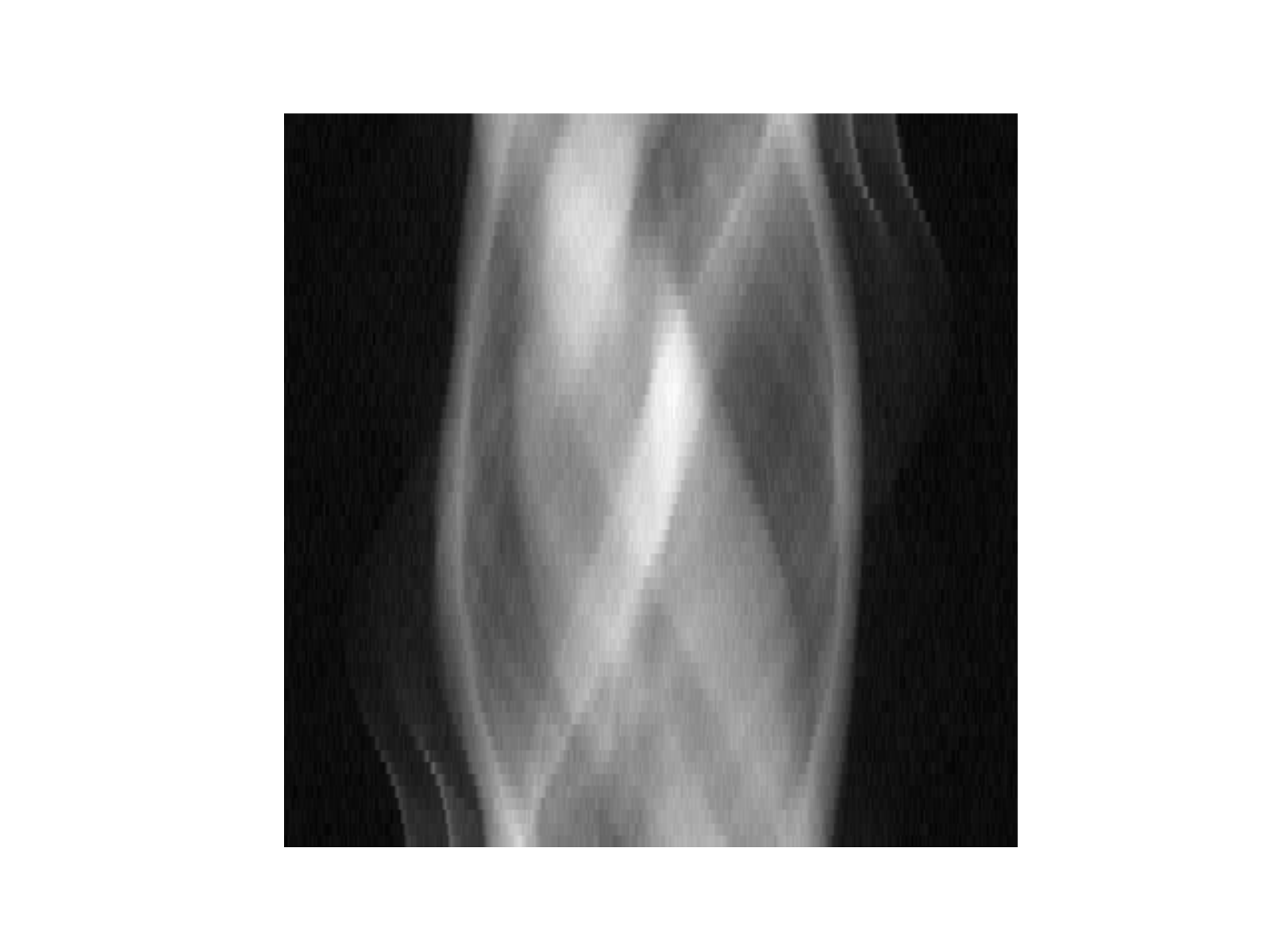}
\end{subfigure}\\
\caption{\textbf{Measurement-data using 60 projection views.} Left: Noise-free data. Right: Noisy data with $5\%$ Gaussian noise added.}
\label{fig:data}
\end{center}
\end{figure}

We minimize \eqref{eq:snett} using the FISTA algorithm \cite{beck2009fast} with a constant step-size. 
We found experimentally that we can use a step-size of at most $s = 10^{-3}$ to guarantee stability. For  the following numerical results we choose  $s = 10^{-3}$   and  use $2000$ iterations to minimize \eqref{eq:snett}. To obtain the initial coefficients we apply the encoder $\encoder_\eta$ to the FBP reconstruction. 
The appropriate regularization parameter value for the DESYRE approach has been chosen empirically to give the best results which resulted in $\al = 10^{-6}$ for the noise-free case and $\al = 3 \cdot 10^{-5}$ for the noisy case. Developing more efficient algorithms for minimizing  the functional  \eqref{eq:snett} is an important aspect of future work, that is beyond  the scope of the present article.

\subsection{Comparison methods}


We compare DESYRE to  wavelet synthesis regularization, 
TV-regularization and a post-processing network. 
In  wavelet synthesis regularization  and  TV-regularization we consider elements 
\begin{align} \label{eq:wave}
	\signal_{\al, \delta}  ^{\rm WS}
	&\in 
	\Ph \kl{\argmin_\xi \norm{\Ko \Ph (\xi) - \data_\delta}^2 +  \alpha \norm{\xi}_1 }
	\\  \label{eq:tv}
	\signal_{\al, \delta}  ^{\text{TV}} 
	&\in \argmin_\signal \norm{\Ko \signal - \data_\delta}^2 + \alpha 
	\norm{\signal}_{\rm TV}  \,, 
\end{align}
respectively.
Here $\Ph$  is the Wavelet synthesis operator corresponding to the Haar Wavelet basis and $\norm{\signal}_{\rm TV}$ is the discrete total variation of $\signal$. 
In order to minimize \eqref{eq:wave} and \eqref{eq:tv} numerically we use the FISTA algorithm \cite{beck2009fast} and the primal-dual algorithm \cite{chambolle2011first}, respectively. The step-sizes for both algorithms are chosen as the inverse of the operator norm of the discretized forward operator. For the Wavelet synthesis regularization we use $500$ iterations, whereas the TV-regularization needs about $4000$ iterations to converge.  For both methods the regularization parameter $\alpha$ was chosen to give the best results which resulted in $\alpha_{\rm WS} = 10^{-8}$ for wavelet synthesis regularization and $\alpha_{\rm TV} = 5 \cdot 10^{-5}$ for TV-regularization in the case of noise-free data and $\alpha_{\rm WS} = 2 \cdot 10^{-7}$ and $\alpha_{\rm TV} = 10^{-4}$ in the case of noisy data.
We initialize each algorithm with $\Ro_{\rm FBP}(\data_{\delta, i})$, where $\Ro_{\rm FBP}(\cdot)$ denotes the filtered back-projection.

As a post-processing network we use the  tight frame U-Net of  \cite{han2018framing}. For a given set of training images $\signal_1, \dots \signal_m$, the network $\Uo_\theta$ is trained to map the filtered back-projection reconstruction $\Ro_{\rm FBP}(\data_{\delta,i})$ to the residual  image   
$ \Ro_{\rm FBP}(\data_{\delta,i}) - \signal_i$. The reconstruction of the signal is then given by $\signal_{\delta}^{\rm Post} = \Ro_{\rm FBP}(\data_{\delta,i}) - \Uo_\theta(\Ro_{\rm FBP}(\data_{\delta,i}))$. 
To obtain images with streaking artefacts the FBP using the Hann filter was applied to the data $\data_i$. No noise was added for the training of the post-processing network.
To regularize the parameters of the network we add $\ell^2$-regularization with regularization parameter $\beta = 10^{-4}$. The network was then trained for $150$ epochs using the Adam optimizer \cite{kingma2014adam} with hyper-parameters as suggested in \cite{kingma2014adam} and a batch size of $3$.

\begin{figure}[htb!]
\begin{center}
\begin{subfigure}{.35\textwidth}
  \centering
  \includegraphics[scale=0.45, trim={3cm 1cm 3cm 1cm},clip]{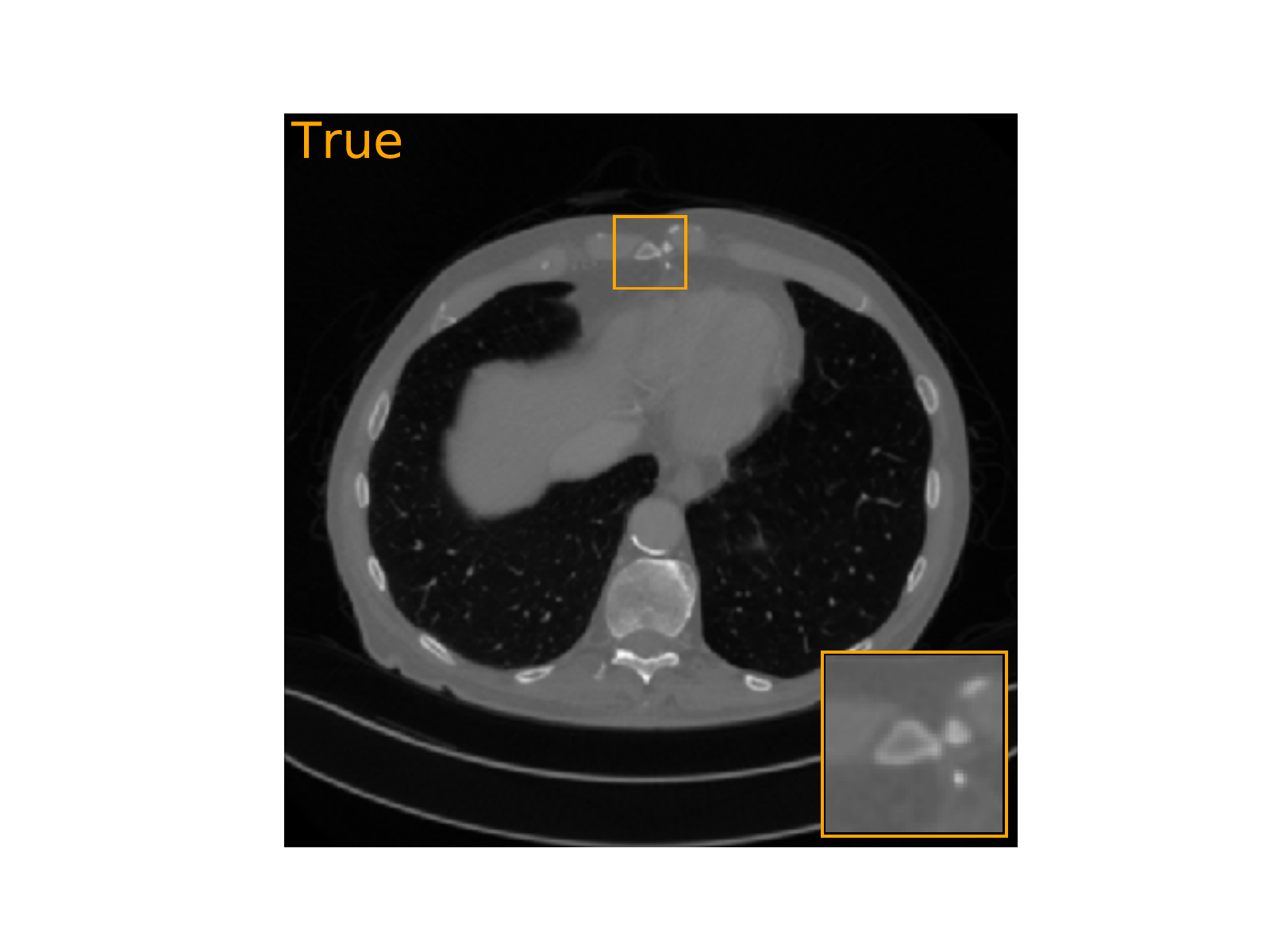}
\end{subfigure}%
\begin{subfigure}{.35\textwidth}
  \centering
  \includegraphics[scale=0.45, trim={3cm 1cm 3cm 1cm},clip]{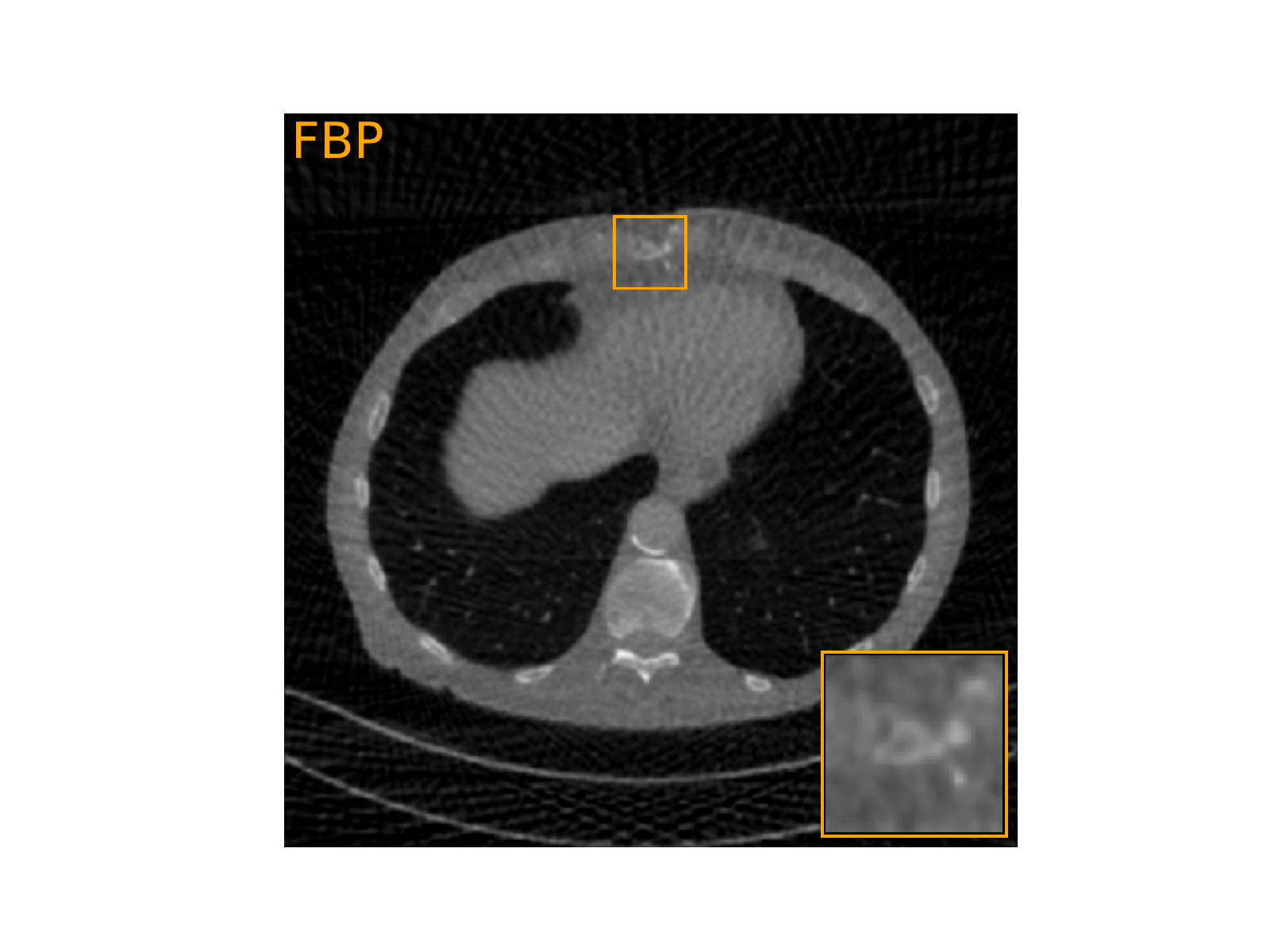}
\end{subfigure}\\

\begin{subfigure}{.35\textwidth}
  \centering
  \includegraphics[scale=0.45, trim={3cm 1cm 3cm 1cm},clip]{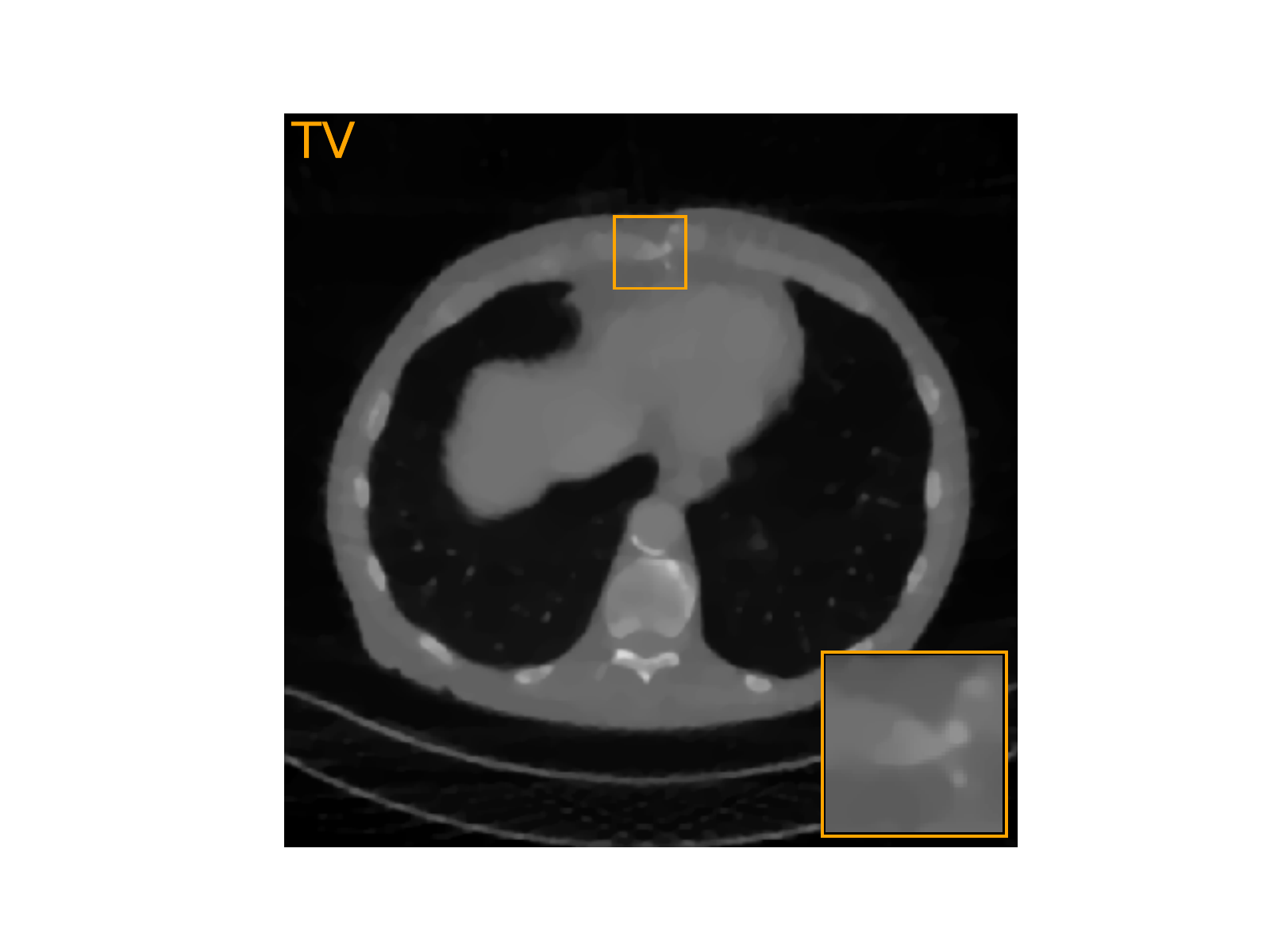}
\end{subfigure}%
\begin{subfigure}{.35\textwidth}
  \centering
  \includegraphics[scale=0.45, trim={3cm 1cm 3cm 1cm},clip]{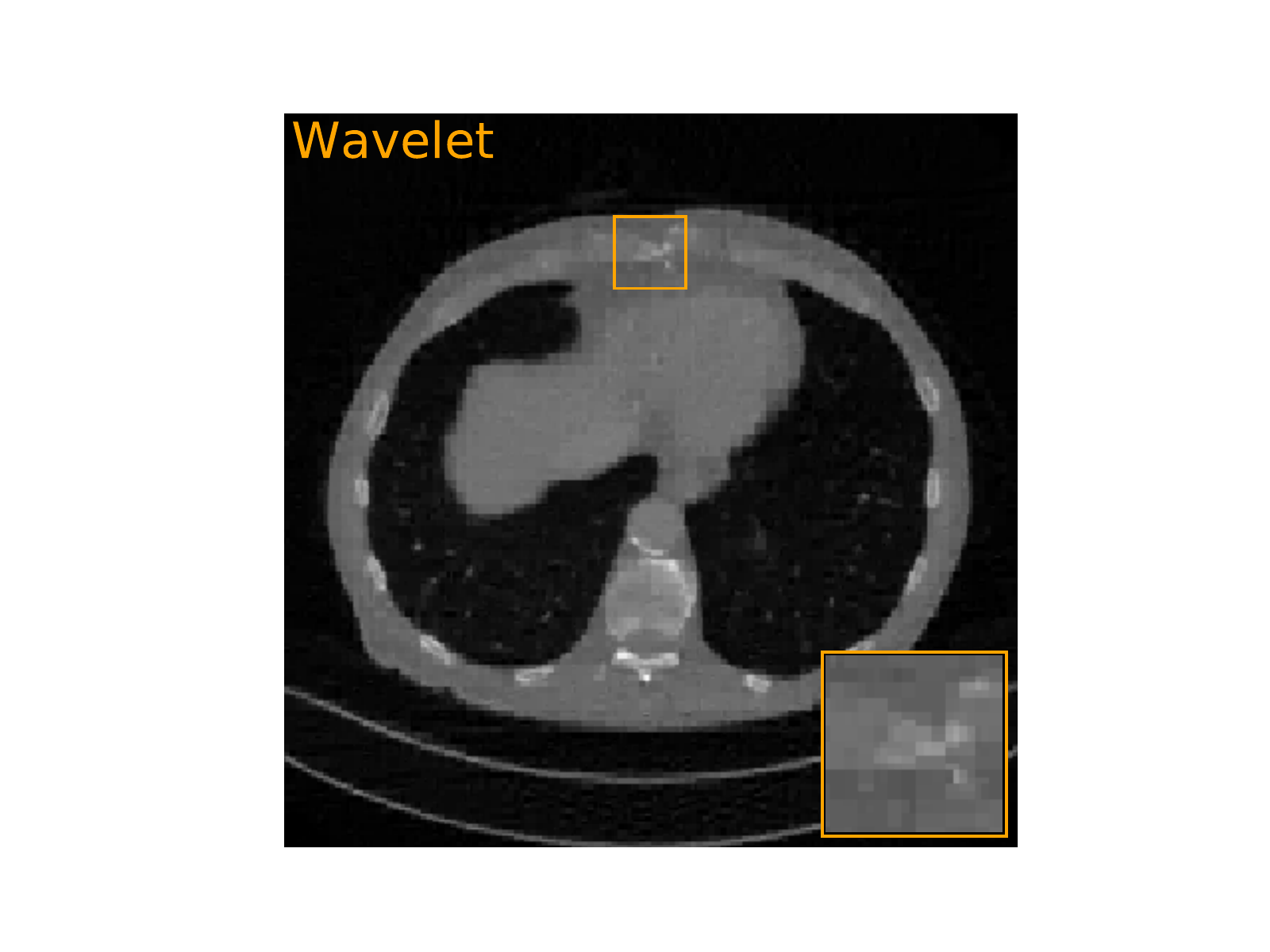}
\end{subfigure}\\

\begin{subfigure}{.35\textwidth}
  \centering
  \includegraphics[scale=0.45, trim={3cm 1cm 3cm 1cm},clip]{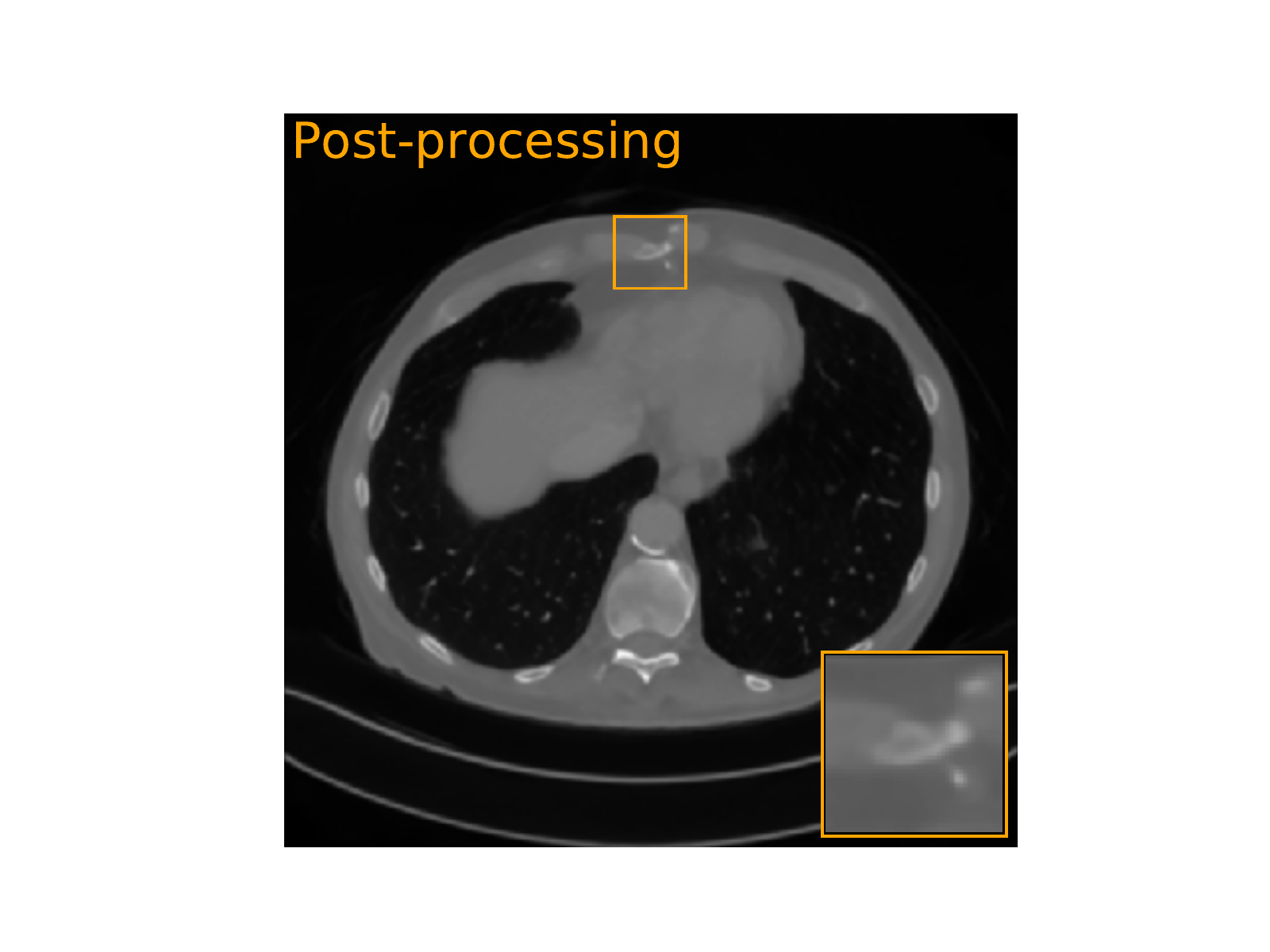}
\end{subfigure}%
\begin{subfigure}{.35\textwidth}
  \centering
  \includegraphics[scale=0.45, trim={3cm 1cm 3cm 1cm},clip]{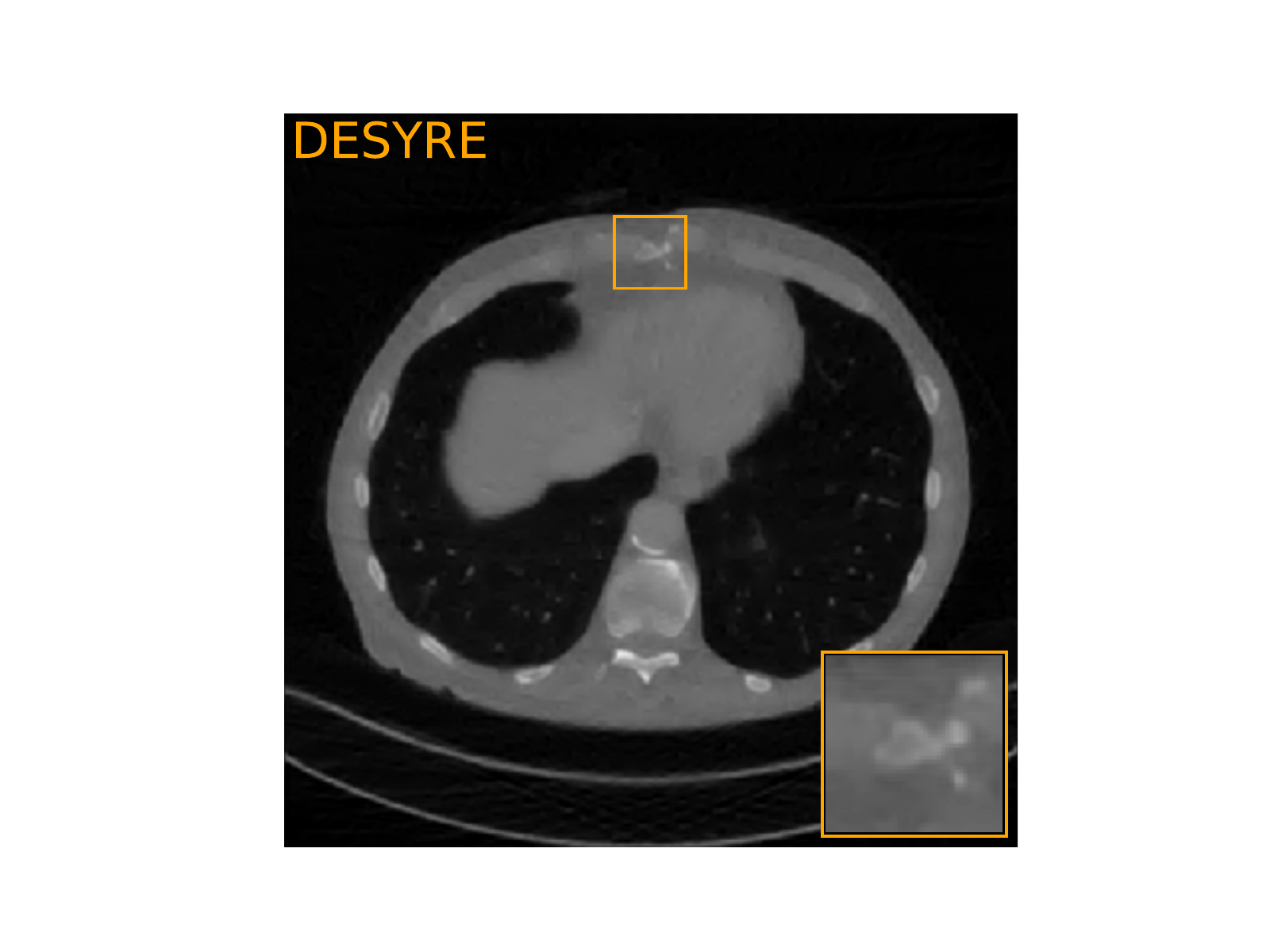}
\end{subfigure}\\
\caption{\textbf{Reconstruction results for $60$ views an noise-free data.} The subplot in the lower right corner shows a zoomed in version of the orange square.}
\label{fig:reconstruction}
\end{center}
\end{figure}

\begin{figure}[htb!]
\begin{center}
\begin{subfigure}{.35\textwidth}
  \centering
  \includegraphics[scale=0.45, trim={3cm 1cm 3cm 1cm},clip]{GROUNDTRUTH.pdf}
\end{subfigure}%
\begin{subfigure}{.35\textwidth}
  \centering
  \includegraphics[scale=0.45, trim={3cm 1cm 3cm 1cm},clip]{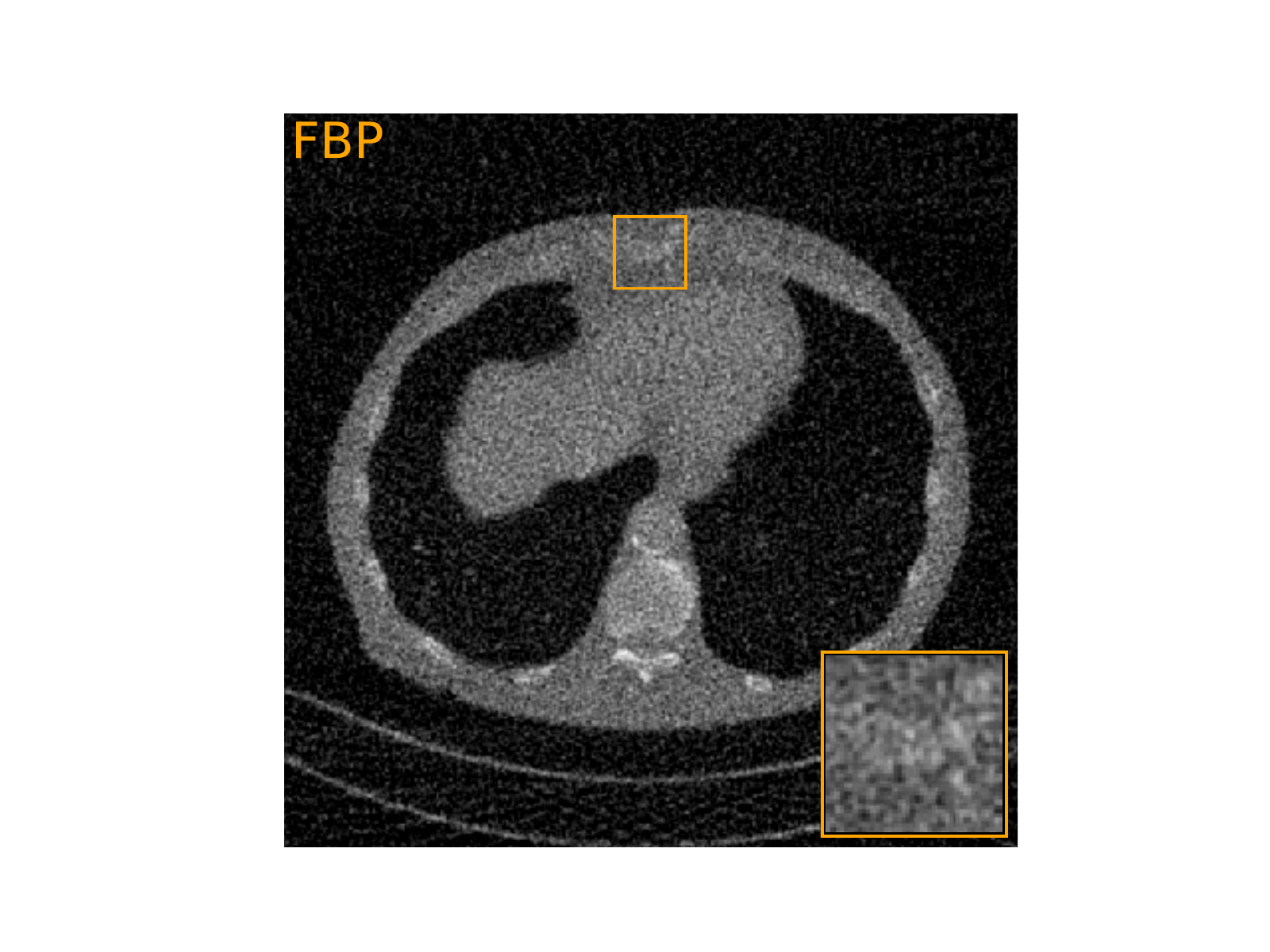}
\end{subfigure}\\

\begin{subfigure}{.35\textwidth}
  \centering
  \includegraphics[scale=0.45, trim={3cm 1cm 3cm 1cm},clip]{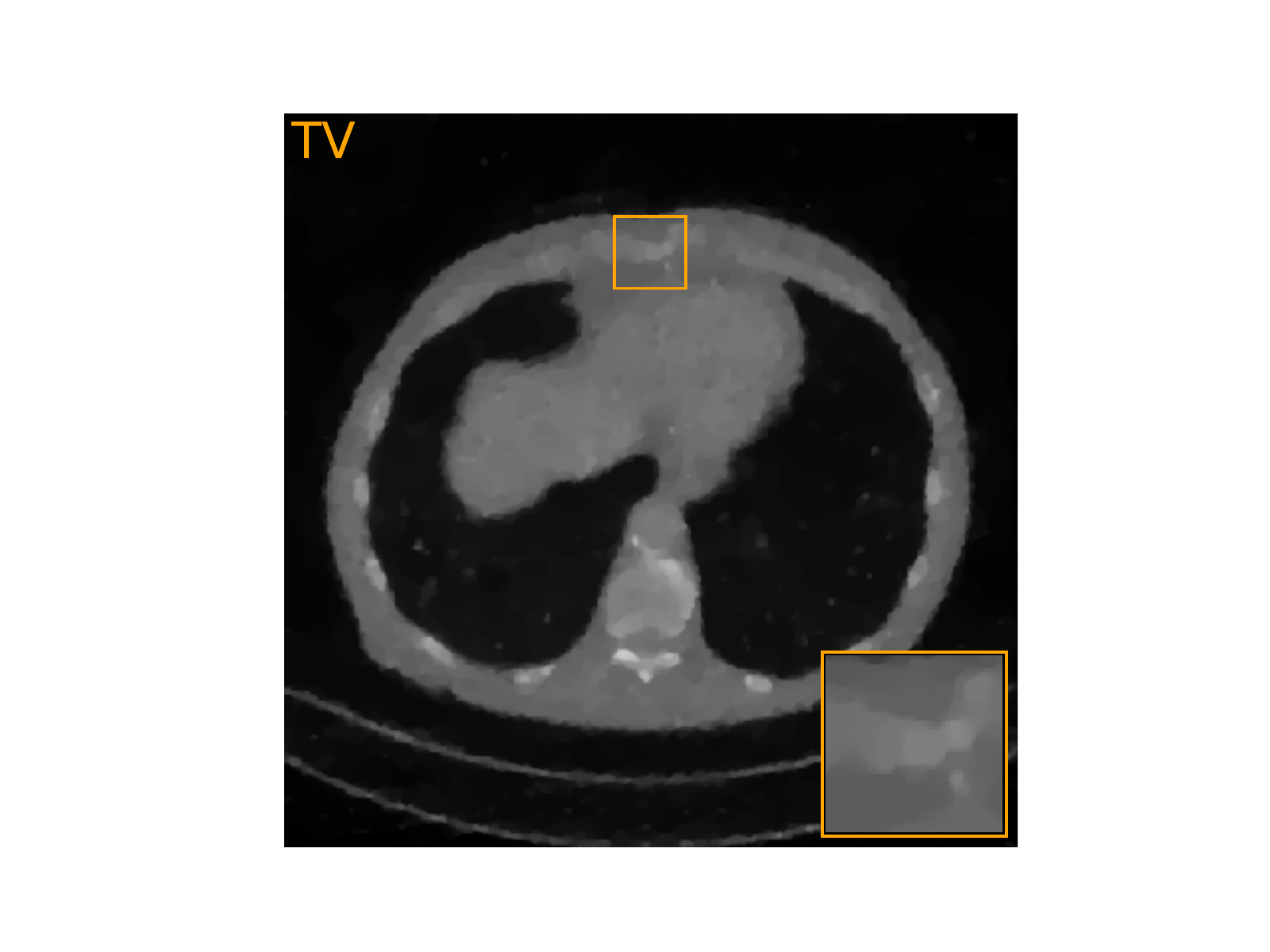}
\end{subfigure}%
\begin{subfigure}{.35\textwidth}
  \centering
  \includegraphics[scale=0.45, trim={3cm 1cm 3cm 1cm},clip]{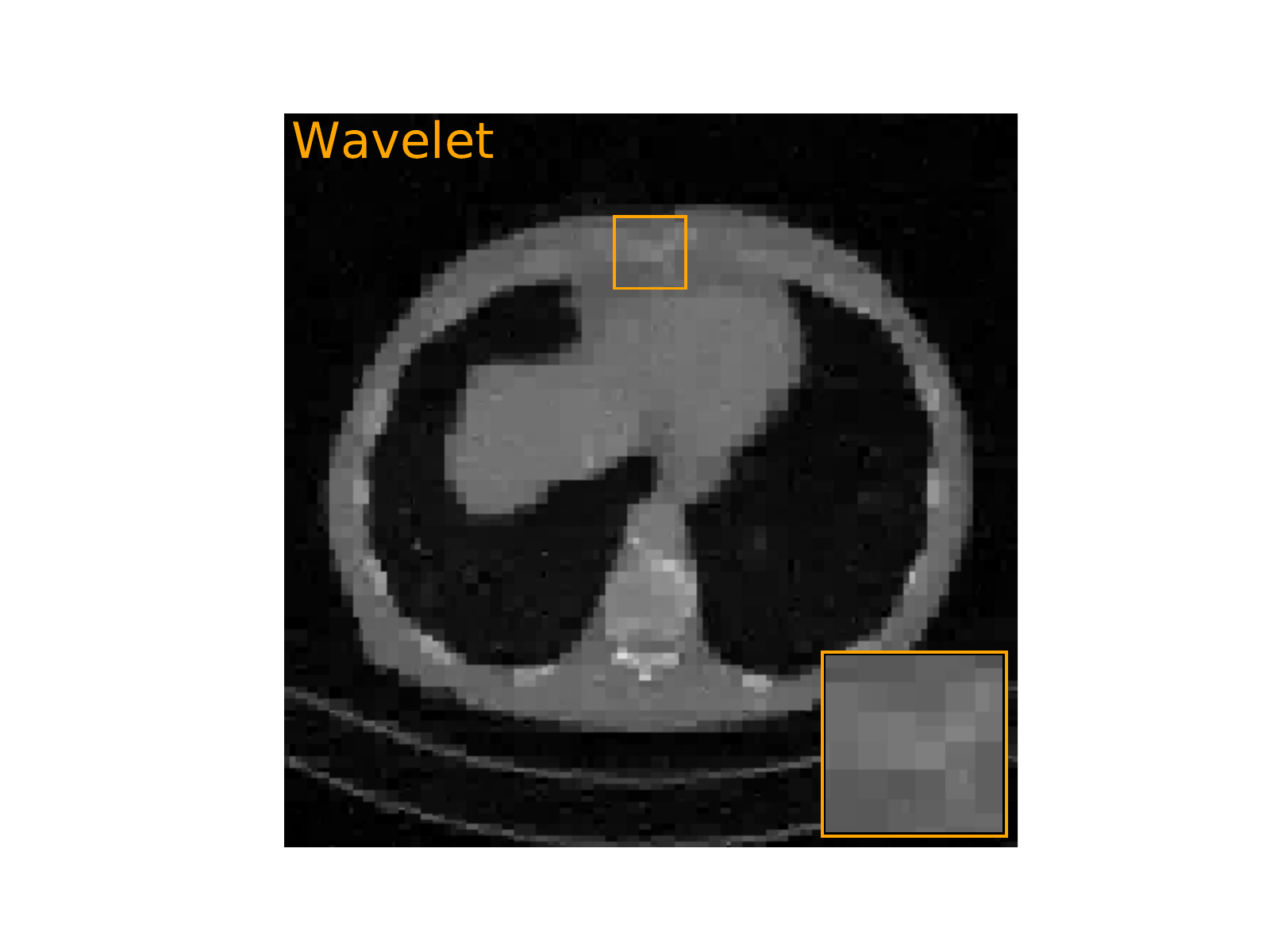}
\end{subfigure}\\

\begin{subfigure}{.35\textwidth}
  \centering
  \includegraphics[scale=0.45, trim={3cm 1cm 3cm 1cm},clip]{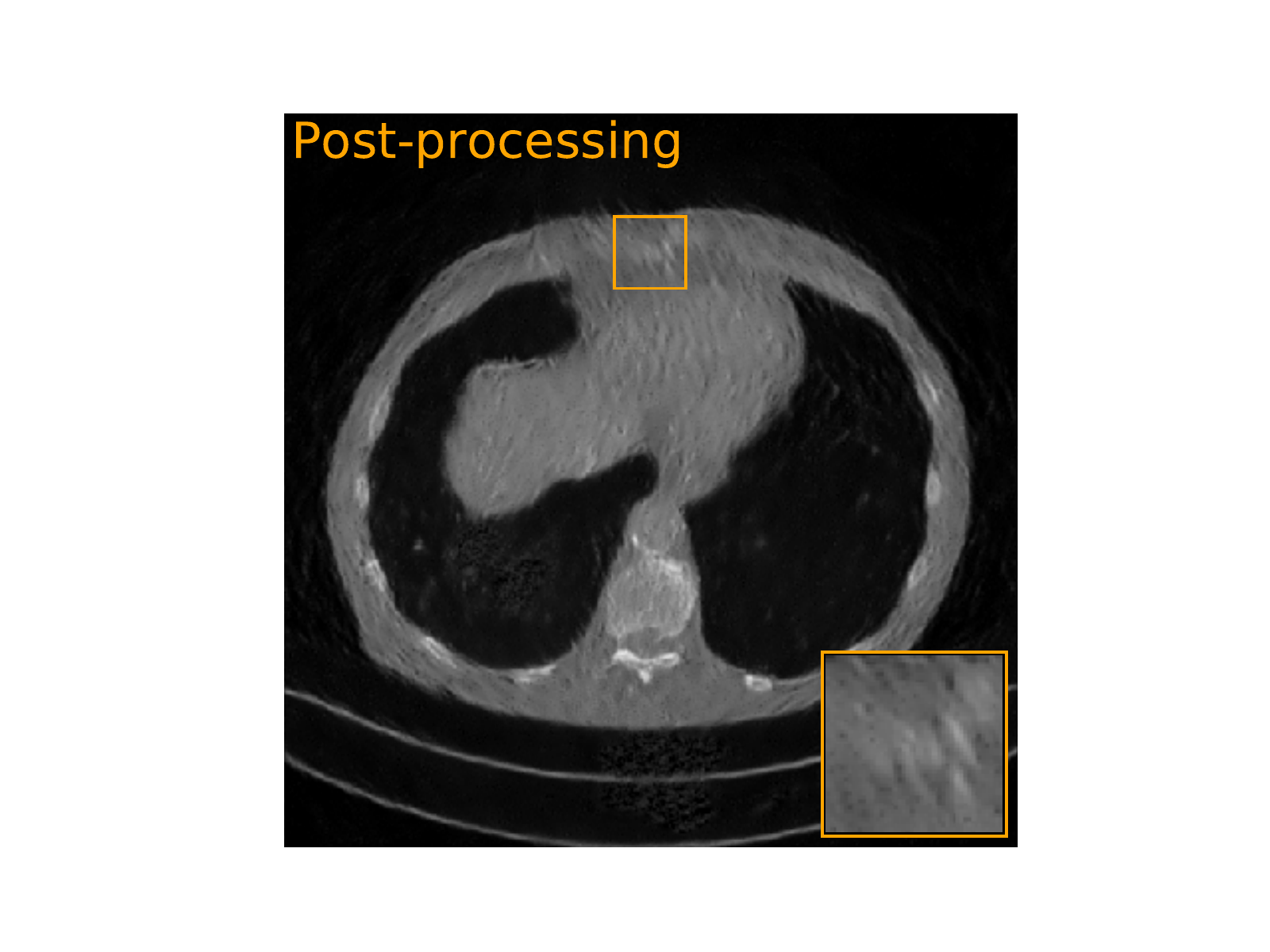}
\end{subfigure}%
\begin{subfigure}{.35\textwidth}
  \centering
  \includegraphics[scale=0.45, trim={3cm 1cm 3cm 1cm},clip]{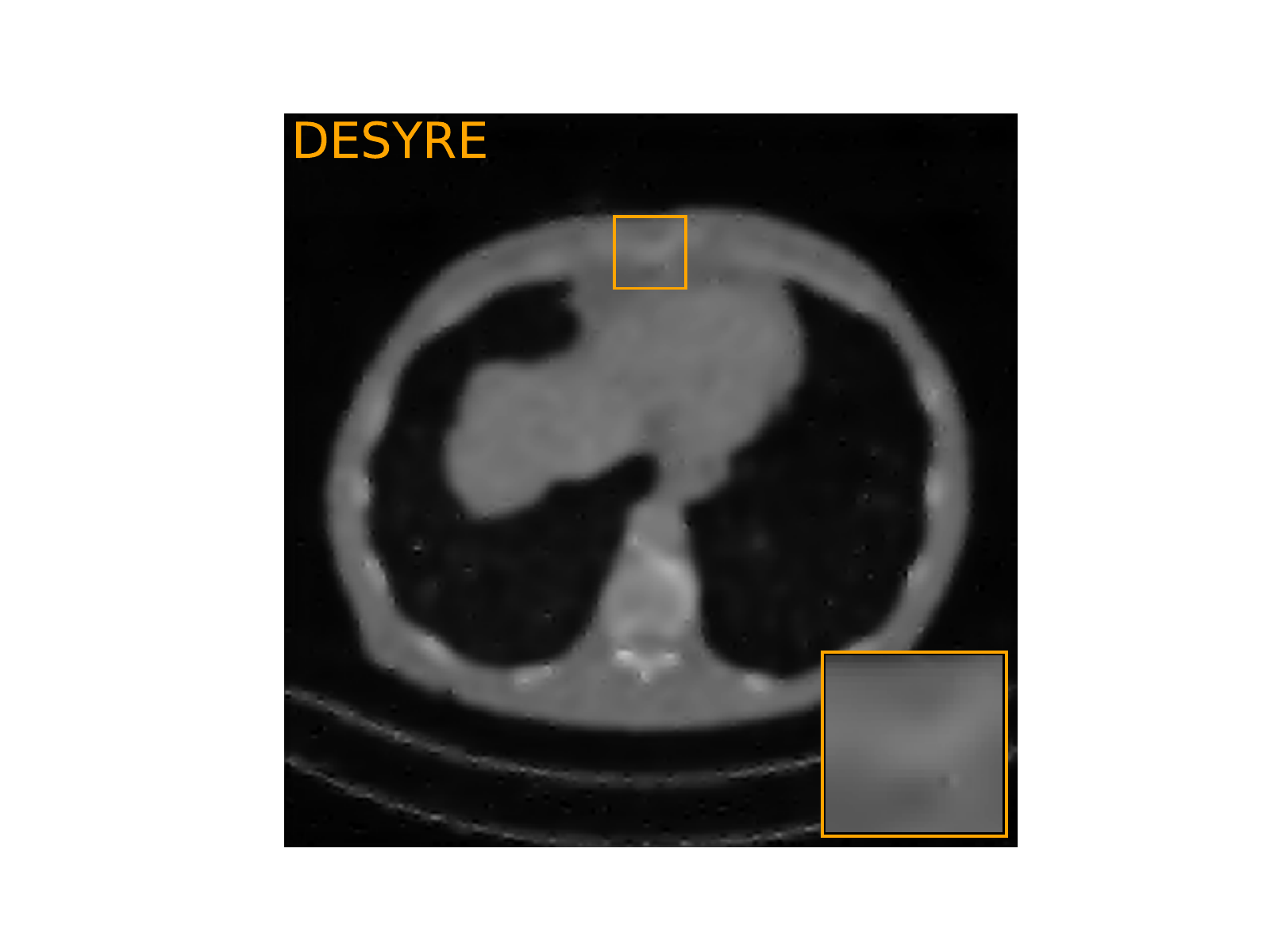}
\end{subfigure}\\
\caption{\textbf{Reconstruction results for $60$ views and $5\%$ Gaussian noise.} The subplot in the lower right corner shows a zoomed in version of the orange square.}
\label{fig:reconstruction_noisy}
\end{center}
\end{figure}

\subsection{Results}

Figure~\ref{fig:reconstruction} shows an example of the different reconstruction methods. We can see that DESYRE is outperformed by the post-processing method. We hypothesize that this is because the post-processing network uses additional information about the inverse problem in training, whereas the training of DESYRE independent of the operator.
In comparison to the other regularization methods, DESYRE shows a better performance visually and quantitatively in the case of noise-free data. When comparing DESYRE with the Wavelet synthesis regularization we see that DESYRE does not suffer from the 'pixel-like' structure even though it is also based on the Haar wavelets. Taking a look at the zoomed in version of the plot, we see that the TV-regularization somehow merges the details, whereas DESYRE is still able to represent these smaller details.  One iteration of the wavelet synthesis, TV regularization and DESYRE take  $0.057$, $0.051$ and  $0.122$ seconds, respectively

\begin{table}[htb!]
\begin{center}
\begin{tabular}{lrr}
\hline
            &  PSNR &  NMSE \\
\hline
FBP         & $28.40 \pm 0.82$  & $0.020 \pm 0.0051$ \\
Wavelet     & $31.07 \pm 0.72$  & $0.011 \pm 0.0029$ \\
TV          & $32.58 \pm 0.60$  & $0.007 \pm 0.0020$ \\
Post-processing         & $39.42 \pm 0.52$  & $0.001 \pm 0.0004$ \\
DESYRE       & $35.49 \pm 0.55$  & $0.004 \pm 0.0009$ \\
 
\hline
\end{tabular}
\end{center}
\caption{\textbf{Quantitative comparison of different reconstruction methods for noise-free data.} Average results $\pm$ standard deviation over 160 different phantoms.}
\label{tab:comparisons}
\end{table}

To quantitatively compare the reconstructions $\signal_{\al, \delta}$  we compute the peak-signal-to-noise-ration and the  normalized-mean-squared-error, respectively 
\begin{align*}
\text{PSNR}(\signal, \signal_{\al, \delta}) &\coloneqq 10 \, \log_{10}\left( \frac{\max (\signal^2)}{\norm{\signal - \signal_{\al, \delta} }^2_2} \right)
\\
\text{NMSE}(\signal, \signal_{\al, \delta}) &\coloneqq \frac{\norm{\signal - \signal_{\al, \delta} }^2_2}{\norm{\signal}_2^2},
\end{align*}
where $\signal$ is the ground truth. The results of these quantitative comparisons in the case of noise-free data can be seen in Table \ref{tab:comparisons}.

Figure~\ref{fig:reconstruction_noisy} shows an example of the reconstructions using noisy data. In this case DESYRE is able to completely remove the noise from the image. However, it cannot satisfactorily recover any small detail features which are present in the image. When compared to TV-regularization DESYRE shows a more blurry image and is outperformed by the TV-regularization. Like in the noise-free case DESYRE shows a smoother but blurrier image. This suggests that additional regularization, e.g. TV-regularization, of the image should be used to further enhance the image quality.
Even though no noise was added during training, the post-processing approach still shows reasonable results. When comparing the two deep learning approaches we see that DESYRE is able to better remove the noise from the image, while the post-processing approach is able to more clearly represent edges.


Lastly we illustrate one practical advantage of DESYRE over a standard post-processing approach. To this end, we consider a similar problem as above. However, this time we only take $n_\theta = 30$ projection views. While the underlying signal class does not change, the reconstruction of the signal is more difficult and the  post-processing network trained with $n_\theta = 60$ has been used. 
This results in reconstructions containing artefacts (Figure~\ref{fig:universality}) which cannot be removed using the post-processing network. Additionally, the post-processing network produces some structure which are  not present in the ground truth.  While the deep synthesis regularization approach was not able to completely remove the artefacts, it was able to greatly reduce them and does not introduce any additional artefacts. While we have not studied this behaviour in great detail, this suggests that the deep synthesis approach is more generally applicable without retraining the network.

\begin{figure}
\begin{center}
\begin{subfigure}{.35\textwidth}
  \centering
  \includegraphics[scale=0.45, trim={3cm 1cm 3cm 1cm},clip]{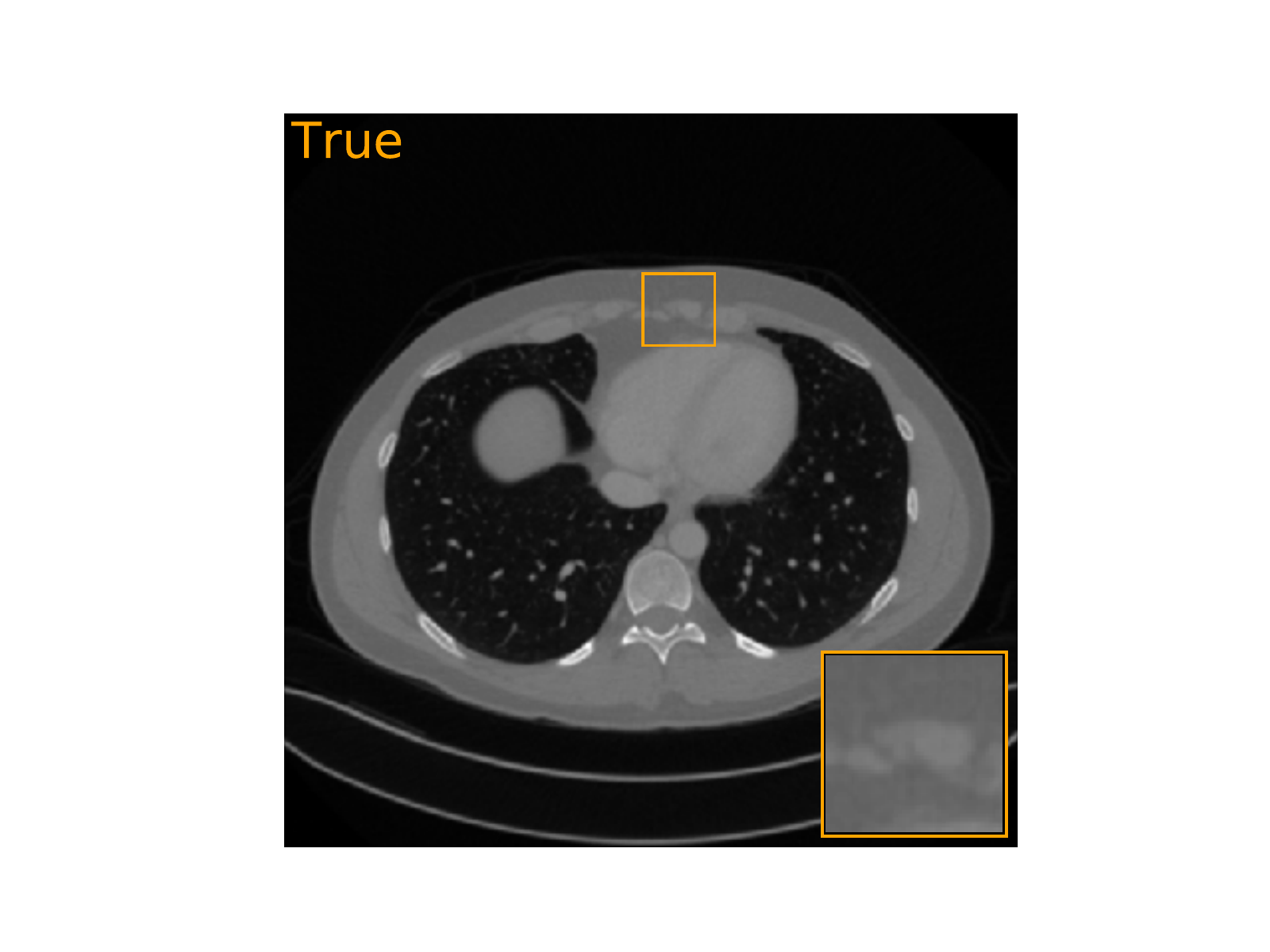}
\end{subfigure}%
\begin{subfigure}{.35\textwidth}
  \centering
  \includegraphics[scale=0.45, trim={3cm 1cm 3cm 1cm},clip]{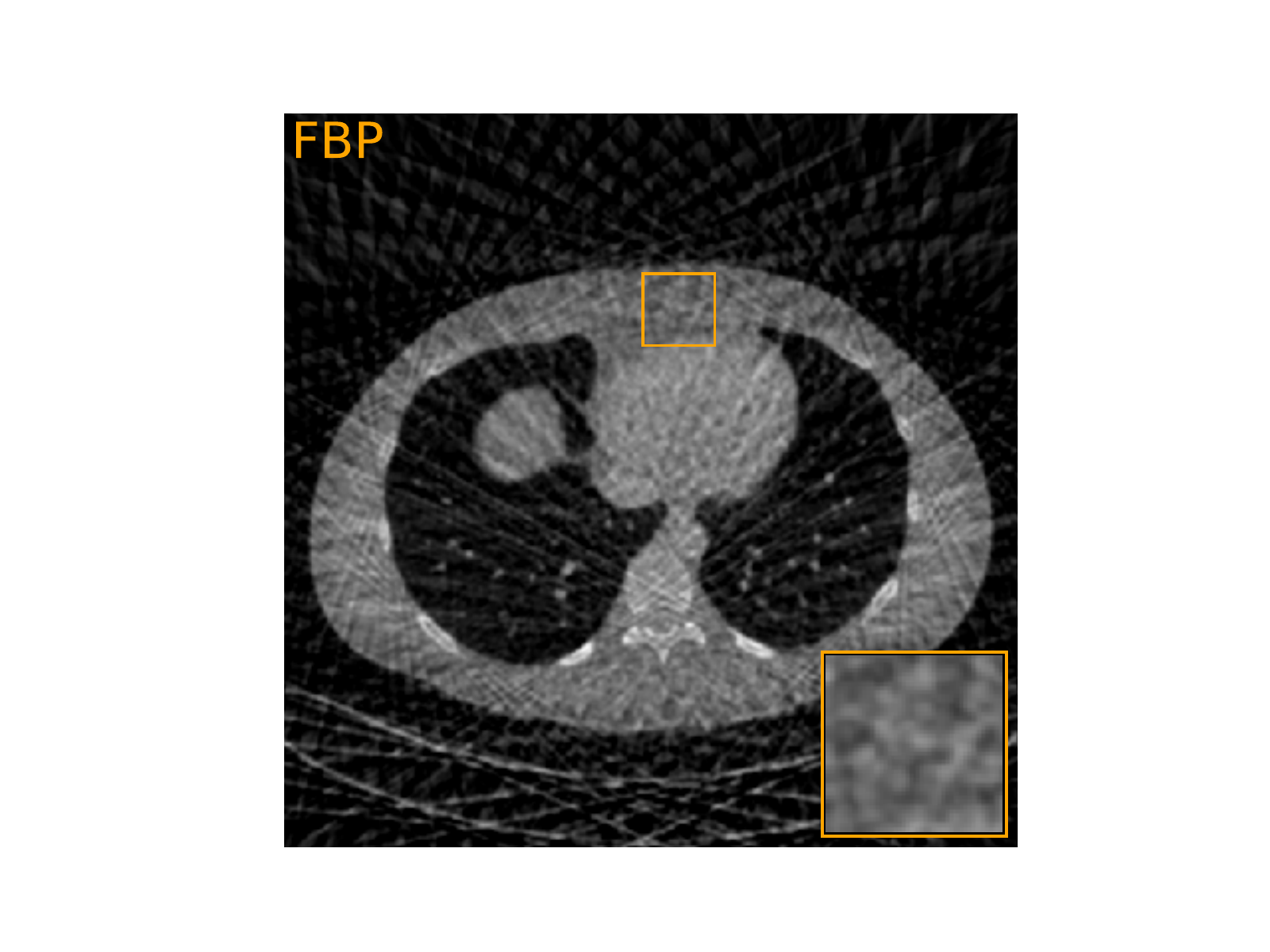}
\end{subfigure}\\

\begin{subfigure}{.35\textwidth}
  \centering
  \includegraphics[scale=0.45, trim={3cm 1cm 3cm 1cm},clip]{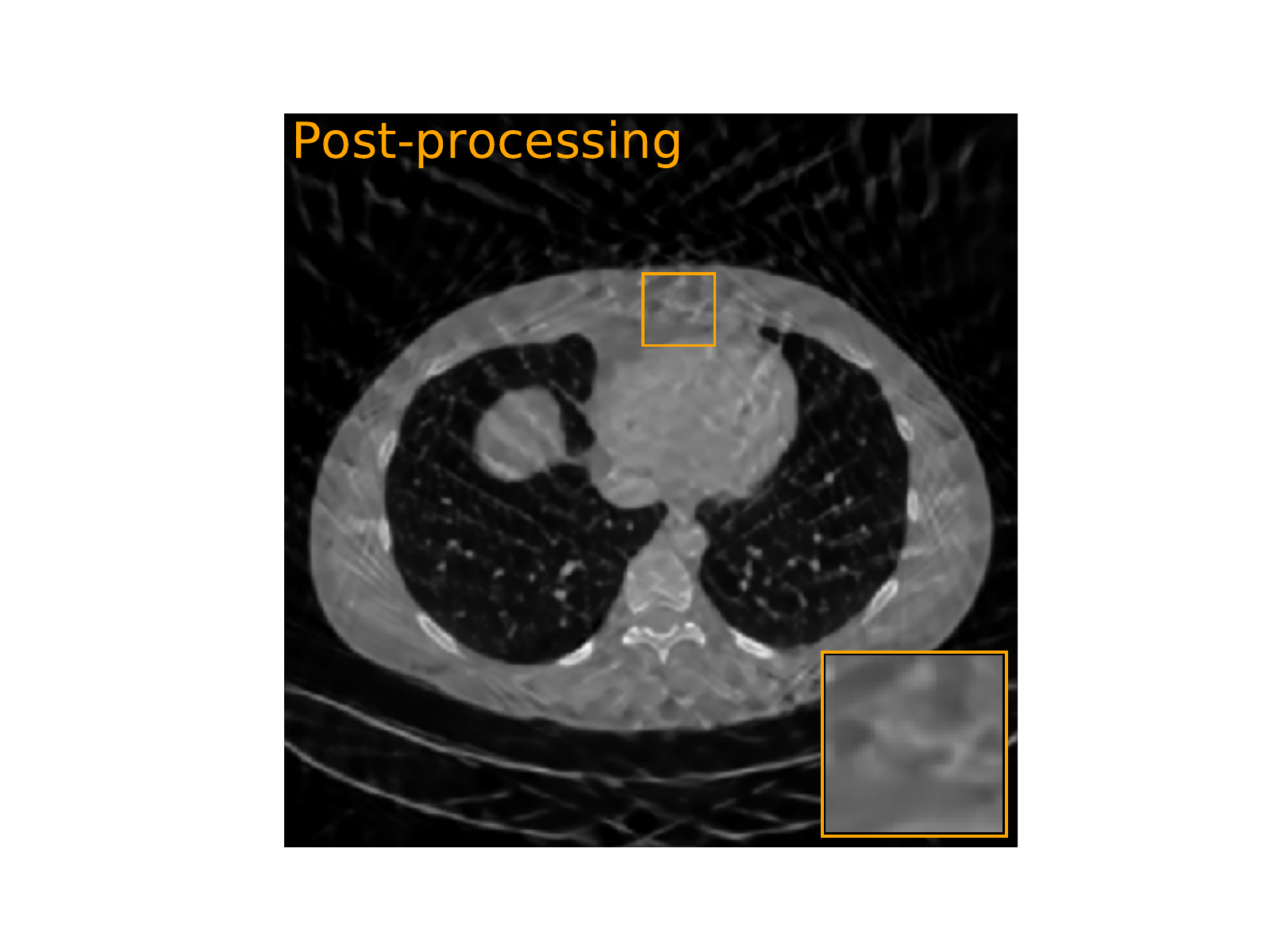}
\end{subfigure}%
\begin{subfigure}{.35\textwidth}
  \centering
  \includegraphics[scale=0.45, trim={3cm 1cm 3cm 1cm},clip]{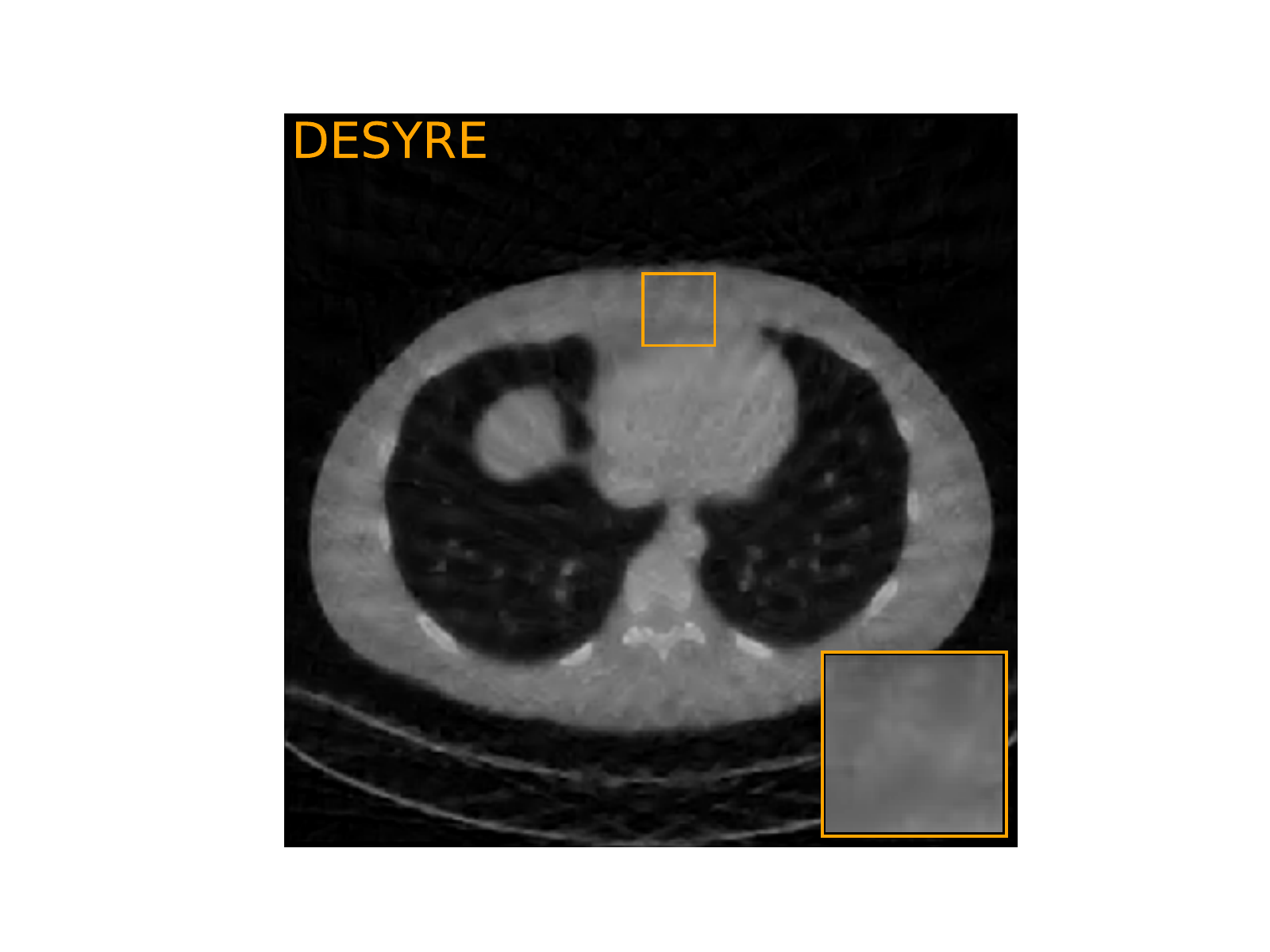}
\end{subfigure}\\
\caption{\textbf{Reconstruction results for $30$ projection views and noise-free data.} The subplot in the lower right corner shows a zoomed in version of the orange square.}
\label{fig:universality}
\end{center}
\end{figure}

\section{Summary and outlook} 
\label{sec:outlook}

We have introduced the deep synthesis   approach for solving  inverse problems. This approach relies on a neural network as a non-linear synthesis operator for representing a signal $\signal$ using the representation $\signal = \decoder (\xi)$ such that $\regg{\xi}$ is small. Using this representation we proposed to solve inverse problems by minimizing the deep synthesis functional \eqref{eq:snett},  generalizing linear frame based methods. 
In section \ref{sec:conv} we proved that the method is indeed a regularization method and we derived linear convergence rates. To find such non-linear  representations  we follow a  data-driven approach where we train a sparse encoder-decoder  pair. We give numerical results of the proposed method and compare it with  other regularization methods and a deep learning approach.

Besides the theoretical benefits, a practical advantage of the deep synthesis  approach over standard post-processing networks is that the training is  independent of the forward operator $\Ko$ and is thus more flexible in changes of the forward operator.
While the results for the sparse view CT problem shows that this approach is outperformed by  specifically trained post-processing network, the deep synthesis  approach outperforms  classical regularization methods for the considered problem.
To further improve the quality of the image reconstruction one could adapt the training strategy to include information about the inverse problem at hand. This could be achieved by, for instance, including images with artefacts in the training data and map these images to some output which does not represent the original image well, thus yielding a large data-discrepancy value in the minimization process.

\section*{Acknowledgements}

D.O. and M.H.  acknowledge support of the Austrian Science Fund (FWF), project P 30747-N32.

%
%

\bibliographystyle{IEEEtran}
\bibliography{paper}
\end{document}